\documentclass[11pt,a4paper]{amsart}
\usepackage{cite}
\usepackage[utf8]{inputenc}
\usepackage{amsmath}
\usepackage{amsfonts}
\usepackage{amssymb}
\usepackage{amsthm}
\usepackage[all, cmtip]{xy}
\usepackage{verbatim}
\usepackage{tikz}
\usepackage{caption}
\usepackage{comment}
\usepackage[shortlabels]{enumitem}

\DeclareMathOperator{\CAT}{CAT}
\DeclareMathOperator{\sing}{sing}
\DeclareMathOperator{\reg}{reg}

\DeclareMathOperator{\join}{Join}
\newtheorem{theorem}{Theorem}[section]
\newtheorem{proposition}[theorem]{Proposition}
\newtheorem{notation}[theorem]{Notation}

\newtheorem{lemma}[theorem]{Lemma}
\newtheorem{remark}[theorem]{Remark}
\newtheorem{example}[theorem]{Example}
\newtheorem{definition}[theorem]{Definition}

\begin{document}

\title{Actions on products of $\mathrm{CAT}(-1)$ spaces}
\author{Teresa García}
\author{Joan Porti}
\address{ Departament de Matem\`atiques, Universitat Aut\`onoma de Barcelona, 
08193 Cerdanyola del Vall\`es, Spain, and 
Barcelona Graduate school of Mathematics (BGSMath) }
\thanks{First author supported by grant
BES-2013-065701. Both authors partially supported by  grant 
FEDER-Meic MTM2015--66165--P}
\email{tggalvez@mat.uab.cat}
\email{porti@mat.uab.cat}
\date{\today}
\maketitle
\begin{abstract}
 We show that for $X$ a proper $\mathrm{CAT}(-1)$ space there is a maximal open subset of the horofunction compactification of $X\times X$ with 
 respect to the maximum metric that compactifies the diagonal action of an infinite quasi-convex group of the isometries of $X$. We also consider 
 the product action of two quasi-convex representations of an infinite hyperbolic group on the product of two different proper $\mathrm{CAT}(-1)$ spaces.  

\end{abstract}

\section{Introduction}
The action of a discrete group $\Gamma$ of isometries on the ideal boundary of a proper  $\mathrm{CAT}(-1)$ space $X$
has a dynamical decomposition $\partial_{\infty} X = \Omega_{\Gamma}\sqcup \Lambda_{\Gamma}$, where 
$\Lambda_{\Gamma}$ is the limit set of the action, and $\Omega_{\Gamma}$ is
the largest open set of the 
ideal boundary where $\Gamma$ acts properly discontinuously  \cite{Coornaert2}. 
In addition, if $\Gamma$ is quasi-convex, then the action on $X\cup \Omega_{\Gamma}$ is also properly discontinuous and cocompact, 
so $\Omega_{\Gamma}$ compactifies the action of $\Gamma$ on $X$  \cite{Swenson}. 


In contrast, for $\mathrm{CAT}(0)$ spaces the dynamical decomposition of the 
visual boundary no longer holds. It may happen, for instance, that there exist 
several maximal subsets of the ideal boundary where the action is properly 
discontinuous or it may also happen that such a set does not exist. There are 
some works in this line, for instance by Papasoglu and Swenson \cite{Papasoglu-Swenson}, or \cite{KLP} in 
the context of symmetric spaces. In this article we consider the case of the 
product of two proper $\mathrm{CAT}(-1)$ spaces, which is  a $\mathrm{CAT}(0)$ space.  
Relevant contributions in these  kind of products are the works of 
 Geniska~\cite{Geniska}   and Link \cite{Link}.

As an example of action on a product of  $\mathrm{CAT}(-1)$ spaces, consider a
cocompact fuchsian group $\Gamma<\mathrm{Isom}(\mathbb{H}^2)$
acting diagonally on 
$\mathbb{H}^2\times\mathbb{H}^2$. The ideal boundary of $\mathbb{H}^2$ is 
$\partial_{\infty}\mathbb{H}^2\cong S^1$, so the visual boundary of the product is the spherical join of two 
circles, $\partial_{\infty}(\mathbb{H}^2\times\mathbb{H}^2) \cong S^1\times 
S^1\times [0, \pi/2]/\!\sim$, where $\sim$ is the relation that collapses each set 
$\{*\}\times S^1\times\{0\}$ and $S^1\times \{*\}\times\{\pi/2\}$ to a point. The 
diagonal action on the ideal boundary preserves the sets $S^1\times 
S^1\times\{\theta\}$ for each $\theta\in[0, \pi/2]$, so 
finding a domain of discontinuity
amounts to find a 
domain of discontinuity for the diagonal action on $S^1\times S^1$. This is not 
possible, since the action on $S^1\times S^1$ has a dense orbit \cite[Thm. 
3.6.1]{Nicholls} and hence admits no domain of discontinuity.

In this example it is worth to notice that the visual compactification of  
$\mathbb{H}^2\times\mathbb{H}^2$ is the horofunction compactification with respect to the product metric, or $\ell^2$ metric.
Instead, here we work with the $\ell^\infty$ or maximum metric, which happens to be better suited for those compactifications.


For $X_1$, $X_2$ two proper $\mathrm{CAT}(-1)$ spaces, we denote their 
horofunction compactification with respect to the max metric by 
$\overline{X_1\times X_2}^{\max}$. It turns out that the ideal boundary of this 
compactification, $\partial_{\infty}^{\max}(X_1\times X_2)$, is homeomorphic to 
the join of the boundaries of each factor. In particular, the  
ideal boundaries for both metrics, $\ell^2$ and $\ell^\infty$, are homeomorphic,  but their 
compactifications are not equivalent, since the 
identity does not extend continuously to the compactifications. 
The max compactification is adapted to diagonal actions, as it allows  to 
find an ideal subset where the diagonal action is properly discontinuous and 
which compactifies the action, the main theorem of this paper being:

\begin{theorem}\label{diagonalcase} Let $X$ be a proper $\mathrm{CAT}(-1)$ space, and $\Gamma$ an infinite quasi-convex group of isometries of $X$.  
There exists an open set $\Omega^{\max}_{\Gamma}\subset \partial_{\infty}^{\max} (X\times X)$ such that:
\begin{enumerate}[label=(\alph*)]
\item The diagonal action of $\Gamma$ on $X\times X\cup \Omega^{\max}_{\Gamma}$ is properly discontinuous and cocompact.
\item $\Omega^{\max}_{\Gamma}$ is the largest open subset of $\partial_{\infty}^{\max} (X\times X)$
 where the diagonal action  is properly discontinuous.
\end{enumerate} 
\end{theorem}

When $\Gamma$ acts cocompactly on $X$, the theorem has been proved in \cite{Tere}.
To prove Theorem \ref{diagonalcase} we show that the nearest point projection from $X\times X$
to the diagonal, by the $\max$ metric, extends continuously to  a map on $\partial_{\infty}^{\max}(X\times X)$
with image on the compactification of the diagonal.


The ideal boundary $\partial_{\infty}^{\max}(X_1\times X_2)$ decomposes in two parts, the regular and the singular one.
The former,  $\partial_{\infty}^{\max}(X_1\times X_2)_{\reg}$, consist of points that correspond 
to the maximum of two Busemann functions, one on each factor, and is 
homeomorphic to 
$\partial_{\infty}X_1\times\partial_{\infty}X_2\times\mathbb{R}$. 
The singular part,  $\partial_{\infty}^{\max}(X_1\times X_2)_{\sing}$,
consists of points 
which are Busemann functions in one of the  factors, and it is  
homeomorphic to $\partial_{\infty}X_1\sqcup\partial_{\infty}X_2$.

In a 
$\mathrm{CAT}(0)$ space the limit set $\Lambda_{\Gamma}$ is the set of 
accumulation points in the ideal boundary of an orbit, and it is independent of 
the choice of the orbit. In our setting, since the $\max$ metric is no longer 
$\mathrm{CAT}(0)$, the set of accumulation points of an orbit depends on the 
orbit, so we consider the \emph{large limit set}, which we define as the union 
of all the accumulation points of each orbit. For a diagonal action it turns out 
that the large limit set is contained in the regular part of the boundary and 
that $\Omega^{\max}_{\Gamma}$ is the complement of the closure of the large 
limit set. In the particular case in which $\Gamma$ is a cocompact group, the 
set $\Omega^{\max}_{\Gamma}$ is naturally homeomorphic to the set of 
parameterized geodesics in one factor, as  shown in \cite{Tere}. 

This $\max$ metric is a Finsler metric. Bordifications through Finsler metrics 
of symmetric spaces have been used by Kapovich and Leeb \cite{KapovichLeeb} to obtain 
a characterization of Anosov representations. In a product
of $\CAT(-1)$ spaces, this
corresponds to the $\ell^1$ metric.
 
\medskip

So far we have seen that the $\max$ compactification is very convenient for 
diagonal actions, but it would be interesting to see in what other cases it is 
useful. For $\Gamma$ an infinite hyperbolic group, we consider $\rho_1$ and 
$\rho_2$ two quasi-convex representations on the isometries of two different 
$\mathrm{CAT}(-1)$ spaces $X_1$ and $X_2$, and their product action 
$\rho_1\times\rho_2$ on
$X_1\times X_2$. 
In analogy to 
the diagonal case, it is reasonable to ask under what conditions the large limit 
set    $\Lambda_{\rho_1\times\rho_2}$ of the product action  also remains inside the regular part of the 
boundary.  We  see that the large limit set being regular implies that  $ 
|d_1(\rho_1(\gamma)o, \rho_1(\gamma')o) - d_2(\rho_2(\gamma)o', 
\rho_2(\gamma')o')| $ is bounded for some $o\in X_1$, $o'\in X_2$
and uniformly for any $\gamma, \gamma'\in \Gamma$. Two representations satisfying this condition are 
said to be \emph{coarsely equivalent}.
As it turns out this  condition  is related to 
the so called marked length spectrum conjecture, since it implies that the 
translation lengths of the two representations are the same. Indeed, we  
prove the following:

\begin{proposition}\label{theoremequivalences} Let $X_1$, $X_2$ be proper $\mathrm{CAT}(-1)$ spaces and $\overline{X_1\times X_2}^{\max}$ the horofunction compactification with respect to $d_{\max}$. Let $\Gamma$ be an infinite hyperbolic group and $\rho_1\colon \Gamma \rightarrow \operatorname{Isom}(X_1)$, $\rho_2\colon \Gamma \rightarrow \operatorname{Isom}(X_2)$ two quasi-convex representations. The following are equivalent:
\begin{enumerate}[label=(\alph*)]
\item $\Lambda_{\rho_1\times\rho_2}\subset \partial_{\infty}^{\max}(X_1\times X_2)_{\reg}.$ 
\item $\rho_1,\,\rho_2  \text{ are coarsely equivalent. } $
\item  $\tau(\rho_1(\gamma)) = \tau(\rho_2(\gamma))$ for all $\gamma\in\Gamma$,
\end{enumerate}
where $\tau(\rho_i(\gamma))$ are the translation lengths for $i=1,2$.
\end{proposition}

If both representations $\rho_1$ and $\rho_2$, apart from being coarsely 
equivalent, are also cocompact, then the spaces $X_1$ and $X_2$ are 
\emph{almost-isometric}. This means that there exists an \emph{almost-isometry} between the spaces, 
which is a quasi-isometry with multiplicative constant one. 
This almost-isometry allows  to construct 
a coarse equivariant map between the regular parts of the ideal boundaries of $X_1\times X_1$ and $X_1\times X_2$,
so that the open set in $\partial_{\infty}^{\max}(X_1\times X_1)$ of Theorem~\ref{diagonalcase} 
is mapped to  an open set $\Omega^{\max}_{\Gamma} \subset \partial_{\infty}^{\max}(X_1\times X_2) $
with good properties:

\begin{theorem}\label{theoremcoarse} Let $X_1$, $X_2$ be proper 
$\mathrm{CAT}(-1)$ spaces and $\overline{X_1\times X_2}^{\max}$ the horofunction 
compactification with respect to $d_{\max}$. Let $\Gamma$ be a hyperbolic group 
and $\rho_1\colon \Gamma \rightarrow \operatorname{Isom} (X_1)$, $\rho_2\colon \Gamma \rightarrow 
\operatorname{Isom} (X_2)$ two cocompact discrete representations. If $\rho_1$ and $\rho_2$ are coarsely 
equivalent, then there exists an open subset $\Omega^{\max}_{\Gamma}\subset 
\partial_{\infty}^{\max}(X_1\times X_2)$ such that the product action of 
$\Gamma$ on $X_1\times X_2\cup \Omega^{\max}_{\Gamma}$ is properly discontinuous 
and cocompact.
\end{theorem}

\section{Preliminaries}

A metric space is said to be \emph{proper} if all its closed balls are compact, and \emph{geodesic} if any two points can be joined by a geodesic segment.

A \emph{$\mathrm{CAT}(-1)$ space} $X$ is a geodesic metric space where triangles are thinner than comparison triangles in the hyperbolic plane. More precisely, for any two points $x,y$ in any geodesic triangle $\Delta\subset X$, and the corresponding comparison points $\overline{x}$, $\overline{y}$ in a comparison geodesic triangle $\overline{\Delta}\subset\mathbb{H}^2$:
$$ d_X(x,y) \leq d_{\mathbb{H}^2}(\overline{x}, \overline{y}).$$
Similarly, a $\mathrm{CAT}(0)$ space satisfies the same condition placing the 
comparison triangles in the euclidean plane. In particular,  $\mathrm{CAT}(-1)$ spaces are also $\mathrm{CAT}(0)$ 
spaces. A good reference for these spaces 
is \cite{BH}.

Two rays $c(t)$, $c'(t)$ in a metric space are said to be \emph{asymptotic} if there exists $C<\infty$ such that
 $d(c(t), c'(t)) \leq C$ for any $t\geq 0$. The \emph{visual 
boundary} $\partial_{\infty}X$ of a metric space $X$ is the set of equivalent 
classes of asymptotic rays. In a proper $\mathrm{CAT}(0)$ space $\overline{X} = X 
\cup \partial_{\infty}X$ can be given a topology (the cone topology, see 
\cite{BH}) such that both $\overline{X}$ and $\partial_{\infty}X$ are compact. 
$\overline{X}$ is known as the \emph{visual compactification}. 

A discrete group action on a topological space $X$ is \emph{properly discontinuous} if every 
compact set intersects finitely many of its translates. For proper metric spaces 
this is equivalent to the fact that every point has an open neighborhood which 
only intersects finitely many of its translates. The action is $\emph{cocompact}$ 
if there exists a compact set $K\subset X$ whose translates cover $X$.  
For $\Gamma$ a 
discrete group of the isometries of a proper $\mathrm{CAT}(0)$ space, the 
\emph{limit set} $\Lambda_{\Gamma}$ is defined as the set of accumulation points 
of an orbit on $\partial_{\infty}X$ and it is independent of the orbit. 
For a $\mathrm{CAT}(-1)$ space $X$, the 
complement of $\Lambda_{\Gamma}$ in $\partial_{\infty}X$ is the \emph{domain 
of discontinuity} $\Omega_{\Gamma}$ and $\Gamma$ acts properly discontinuously on 
$\Omega_{\Gamma}$  \cite{Coornaert2}.

The \emph{quasi-convex hull} of a set $S\subset\overline{X}$ is the union of all 
segments, rays, and geodesic lines of $X$ with endpoints in $S$. A set $S\subset 
\overline{X}$ is \emph{quasi-convex} if $S\cap X$   is contained in an epsilon 
neighborhood of its weak convex hull. A group $\Gamma$ of the isometries of a 
$\mathrm{CAT}(-1)$ space $X$ is \emph{quasi-convex} if it acts properly 
discontinuously on $X$ and any orbit is a quasi-convex set.

A \emph{quasi-isometric embedding} between metric spaces $X$, $Y$ is a map $f\colon X \rightarrow Y$ such that there exists constants $A\geq 1$ and $C\geq 0$ such that for all $x_1, x_2\in X$:
$$\frac{1}{A}d_X(x_1, x_2) - C \leq d_Y(f(x_1), f(x_2)) \leq Ad_X(x_1, x_2) + C.$$
If the map $f$ is also \emph{coarsely onto}, i.e. for each $y\in Y$ there exists $x\in X$ with:
$$d_Y(f(x), y) \leq C,$$
then it is a \emph{quasi-isometry}. An \emph{almost-isometry} is a quasi-isometry with multiplicative constant $A = 1$.

A quasi-convex group of isometries of a proper $\mathrm{CAT}(-1)$ space $X$ is hyperbolic and finitely generated. Moreover, the orbit map:
\begin{align*}
 \Gamma &\rightarrow X\\
 \gamma &\mapsto \gamma o
\end{align*}
is a quasi-isometric embedding for any $o\in X$ and it extends to an equivariant 
homeomorphism (which is also Lipschitz and quasi-conformal) from 
$\partial_{\infty}\Gamma$ to its limit set $\Lambda_{\Gamma}$, \cite{Bourdon}. 
The action of a quasi-convex group on $X\cup \Omega_{\Gamma}$ is properly 
discontinuous \cite{Coornaert2} and cocompact \cite{Swenson}.

\section{The max compactification}
\label{compactification}

Let $(X_1, d_1)$, $(X_2, d_2)$ be two proper $\mathrm{CAT}(-1)$ spaces, and consider the product space $X_1\times X_2$, together with the $\max$ metric $d_{\max}$ (or $\ell^{\infty}$ metric), which is defined as:
$$d_{\max}((x,y),(x',y'))= \max\{d_1(x,x'), d_2(y,y')\}$$
for any $(x,y)$, $(x',y')$ in $X_1\times X_2$. 
The metrics $d_{\max}$ and the product metric:
$$d_{\ell^2}((x,y),(x',y'))=\sqrt{d_1(x,x')^2 + d_2(y,y')^2}$$ are comparable so they induce the same topology in $X_1\times X_2$. 
For $X_1$ and $X_2$  proper geodesic spaces, $(X_1\times X_2, d_{\max})$ is also a proper geodesic space 
\cite[Prop.2.6.6]{Papadopoulos}. 
In this section we compute the horofunction compactification of $X_1\times X_2$ with respect to the metric $d_{\max}$.  

Let $X$ be a proper metric space and $C_*(X)$ its space of continuous functions 
(with the topology of uniform convergence on compact sets) modulo the additive 
action of the subspace of constant functions. The \emph{Gromov} or \emph{horofunction compactification} 
of $X$ is the closure in $C_*(X)$ of the image of the map:
 \begin{align*}
\iota\colon  X &\rightarrow C_*(X)\\
 x &\mapsto [d(x, \cdot)],
\end{align*}
which assigns to each point  $x$ in $X$  the class in $C_*(X)$ of the distance function with respect to this point, 
see \cite{BallmannGromov}. We denote the horofunction compactification of $X$ by $\overline{X}$. The ideal boundary, denoted by $\partial_{\infty}X$, is the set $\overline{X}\setminus \iota(X)$. For a proper metric space both  $\overline{X}$ and  $\partial_{\infty}X$ are compact and metrizable spaces. 

\begin{remark}\label{convergence}
Fixing a base point $o\in X$, the sequence $[d(x_n, \cdot)]$  converges to a class of functions $[f]\in C_{*}(X)$ if and only if the sequence of corresponding normalized distance functions, $d(x_n, \cdot) -d(x_n, o)$, converges to $f-f(o)$ uniformly on all balls $B(o, r)$. In fact, $C_{*}(X)$ is homeomorphic to the subspace of $C(X)$ of continuous functions with $f(o)=0$.
\end{remark}
\begin{definition}
A \emph{horofunction} $h$ is a continuous function in $C(X)$ such that its class $[h]$ belongs to $\partial_{\infty}X$. 
\end{definition}
\begin{remark}
We will usually call a class of horofunctions an ideal point and denote it by $\xi$. Then the horofunction $h$ in the class $\xi$ satisfying $h(o)=0$ is denoted by $h^o_{\xi}$.
\end{remark}

\begin{notation} When we say that a sequence $(x_n)_n$ converges to an ideal point $\xi$ in the horofunction compactification, $x_n \rightarrow \xi$, we mean that for a base point $o\in X$ the corresponding sequence of normalized distance functions converges uniformly on compact sets to the horofunction $h^o_{\xi}$.
\end{notation}

The level sets of a horofunction are known as \emph{horospheres} and the sublevel sets as \emph{horoballs}. Observe that two horofunctions in the same equivalence class differ by a constant and share the same set of horospheres and horoballs. The horofunctions of a proper $\text{CAT}(0)$ space are Busemann functions:  

\begin{definition}
A \emph{Busemann} function in a metric space $(X,d)$ is a function defined as:
$$z \mapsto \lim_{t\rightarrow+\infty} d(c(t), z) - t$$
for some geodesic ray $c(t)$ in $X$. 
\end{definition}

In a proper $\mathrm{CAT}(0)$ space $X$, given a point $o\in X$ and an ideal point $\xi\in\partial_{\infty}X$ there is a unique 
ray $c(t)$ such that $c(0)=o$ and its associated Busemann function is in the class $\xi$. This Busemann function is denoted by
$$
\beta^o_{\xi}(z)=\lim_{t\rightarrow+\infty} d(c(t), z) - t  ,
$$
%
The horofunction compactification and the visual compactification of a proper $\mathrm{CAT}(0)$ space are equivalent \cite[Cor.~8.20]{BH}. 

\begin{lemma}[cf.~\cite{BH}]
\label{lemma:propertiesbusemann}
For a $\CAT(0)$ space $X$:
 \begin{enumerate}[label=(\roman*)]
 \item If $\sigma\colon[0,+\infty)\to X$ is a ray in the class $\xi\in\partial_\infty X$,
then $\beta^o_\eta(\sigma(s))$ converges to $+\infty$ if $\eta\neq\xi$ and to $-\infty$ if $\eta = \xi$.
\item For any $p,q\in X$,   $\beta^p_\eta-\beta^q_\eta$ is a constant function. 
\end{enumerate}
\end{lemma}

The Gromov or  horofunction compactification of $(X_1 \times X_2, d_{\max})$ is denoted by $\overline{X_1\times X_2}^{\max}$ and its ideal boundary by $\partial_{\infty}^{\max}(X_1\times X_2)$.
We choose a base point $O=(o, o')$ with $o \in X_1$ and $o'\in X_2$ and then, as a representative of a class of normalized distance functions, we have the function: 
\begin{eqnarray*}
d_{\max}^{O}((x,y), \cdot) &=& d_{\max}((x,y), \cdot) - d_{\max}((x,y), (o,o'))\\ &=& \max\{d_1(x,\cdot), d_2(y, \cdot)\} - \max\{d_1(x,o), d_2(y,o')\}.
\end{eqnarray*}
Then, by Remark \ref{convergence}, $[d_{\max}((x_n, y_n), \cdot)]\rightarrow \xi\in\partial_{\infty}^{\max}(X_1\times X_2)$ if and only if $d_{\max}^{O}((x_n,y_n), \cdot) \rightarrow h^{O}_{\xi}$, where $h^{O}_{\xi}$ is the horofunction in the class $\xi$ that satisfies $h^{O}_{\xi}(O)=0$.

Given a diverging sequence  $(x_n, y_n) \subset (X_1 \times X_2, d_{\max})$,
we distinguish two cases, up to subsequence:
\begin{enumerate}[label=(\Roman*)]
 \item either $\vert d_1(x_n, o) - d_2(y_n, o') \vert\to \infty$,
 \item or $\vert d_1(x_n, o) - d_2(y_n, o') \vert$ remains bounded.
\end{enumerate}
Notice that if one of $d_1(x_n, o)$ or $d_2(y_n, o') $ is bounded, then we are in the first case, as we assume that  $(x_n, y_n) $     diverges.

\begin{proposition}
 \label{boundary points}
Let $(x_n, y_n) \subset (X_1 \times X_2, d_{\max})$ be a diverging sequence.  
\begin{enumerate}[label=(\Roman*)]
 \item If $\vert d_1(x_n, o) - d_2(y_n, o') \vert\to \infty$, then, up to a subsequence and permuting the coordinates,  there exists $\xi\in\partial_\infty X_1$ such that
 $$\lim_{n\rightarrow \infty} d_{\max}^{O}((x_n,y_n),(z,z'))  = \beta^{o}_{\xi}(z).$$
  \item If $\vert d_1(x_n, o) - d_2(y_n, o') \vert$ remains bounded, then, up to a subsequence,  there exist $\xi\in\partial_\infty X_1$ and $\xi'\in\partial_\infty X_2$ such that
  $$\lim_{n\rightarrow \infty} d_{\max}^{O}((x_n,y_n),(z,z'))=  
      \max\{\beta^{o}_{\xi}(z), \beta^{o'}_{\xi'}(z')-C\},
  $$ for some $C\in\mathbb{R}$.
\end{enumerate}
\end{proposition}


\begin{proof}
We do the proof for case (II), the other  case can be obtained in a similar fashion. For each $n$, denote $C_n= d_1(x_n,o) - d_2(y_n,o')$ and assume 
(after permuting the factors) that
$C_n\geq 0$. Then, $d_{\max}^{O}((x_n,y_n),(z,z'))$ can be rewritten as:
\begin{multline*}
d_{\max}((x_n, y_n), (x,y))- d_{\max}((x_n,y_n), (o,o')) =\\  \max\{d_1(x_n,x) - d_1(x_n,o), d_2(y_n,y)-d_2(y_n,o')-C_n\}.
\end{multline*}
Using that the space is proper, that the maximum is continuous, and that $d_1(x_n,x) - 
d_1(x_n,o) \rightarrow \beta^o_{\xi} $ and
there exist $\xi\in\partial_\infty X_1$ and $\xi'\in\partial_\infty X_2$ such that 
$d_2(y_n,y)-d_2(y_n,o')-C_n 
\rightarrow \beta^{o'}_{\xi'}-C $ uniformly on compact sets of $X_1$ and $X_2$, 
respectively:
\begin{multline*}
\max\{d_1(x_n,x) - d_1(x_n,o), d_2(y_n,y)-d_2(y_n,o')-C_n\} \rightarrow\max\{\beta^o_{\xi}, \beta^{o'}_{\xi'}-C\}
\end{multline*}
uniformly on compact sets of $X_1\times X_2$.
\end{proof}
We call the set of ideal boundary points in $\partial^{\mathrm{max}}_\infty (X_1\times X_2$
with a representative of the form $\beta^{o}_{\xi}(z)$ or $\beta^{o'}_{\xi'}(z')$, i.e.~case (I), the 
\emph{singular part} of the boundary and we denote it by $\partial_{\infty}^{\max}(X_1 \times X_2)_{\sing}$. 
The rest of the ideal boundary  points, those with a representative of the form $\max\{\beta^{o}_{\xi}(z), \beta^{o'}_{\xi'}(z')+C\}$ with $C\in\mathbb{R}$, i.e.~case (II), define the \emph{regular part} of the boundary and we 
denote this set by $\partial_{\infty}^{\max}(X_1 \times X_2)_{\reg}$.

Using that the set of Busemann functions in one factor is naturally identified to the boundary of this factor, we have:

\begin{proposition}
\label{homeo singular}
There is a natural homeomorphism 
$$\varphi_{\sing}\colon  \partial_{\infty}^{\max}(X_1 \times X_2)_{\sing} \longrightarrow \partial_{\infty}X_1 \sqcup \partial_{\infty} X_2 $$
that consists in associating to a Busemann function that takes values only in the first (second) factor of $X_1 \times X_2$ the same Busemann function viewed as a point of the first (second)  
summand in $\partial_{\infty}X_1 \sqcup \partial_{\infty} X_2$.
\end{proposition}

For the regular part, notice that we can get rid of the additive constant in Proposition~\ref{boundary points} by changing the base point. Thus
 regular points are  the 
classes modulo constant of the functions $\max\{\beta^p_{\xi}(z), 
\beta^{p'}_{\xi'}(z')\}$ for all $p\in X_1$, $p'\in X_2$, $\xi\in 
\partial_{\infty}X_1$ and $\xi'\in \partial_{\infty}X_2$.

\begin{proposition}
\label{homeo regular}
For each choice of base point $(o,o')\in X_1\times X_2$ there is a natural homeomorphism 
\begin{equation}
\label{homeoreg}
\begin{array}{r@{}l}
\varphi_{\reg}\colon  \partial_{\infty}^{\max}(X_1 \times X_2)_{\reg} &{}\longrightarrow  \partial_{\infty} X_1 \times \partial_{\infty} X_2\times \mathbb{R} \\ 
\left[ \max\{\beta_{\xi}^{p}(z), \beta^{p'}_{\xi'}(z')\}\right]  &{} \mapsto (\xi, \xi', \beta^{p}_{\xi}(o)-\beta^{p'}_{\xi'}(o')).
\end{array}
\end{equation}
\end{proposition}
\begin{remark}
\label{homeo regular with a base point}
If we fix $p=o$ and $p'=o'$,  then homeomorphism \eqref{homeoreg} takes the form:
 \begin{equation}
\label{homeobasepoint}
\begin{array}{r@{}l}
\varphi_{\reg}\colon  \partial_{\infty}^{\max}(X_1 \times X_2)_{\reg} &{}\longrightarrow  \partial_{\infty} X_1 \times \partial_{\infty} X_2\times \mathbb{R} \\ 
\max\{\beta^{o}_{\xi}(z), \beta^{o'}_{\xi'}(z')-C\} &{}\mapsto (\xi, \xi', C)
\end{array}
\end{equation}
where $C\in\mathbb R$.
\end{remark}
\begin{proof}[Proof of Prop.~\ref{homeo regular}]

To prove that the map \eqref{homeobasepoint} is well defined, consider rays $c\colon [0,+\infty)\to X_1$   with $c(0)=p$ and 
 $c'\colon [0,+\infty)\to X_2$  with $c'(0)=p'$. Then, by   Lemma~\ref{lemma:propertiesbusemann}
$\lim_{t\to{+\infty}}\max\{\beta_{\xi}^{p}(c(t)), \beta^{p'}_{\xi'}(c'(t))\}=-\infty$ if and only if $c(+\infty)=\xi$ and $c'(+\infty)=\xi'$;
otherwise this limit is $+\infty$. 
From this property it easily follows that $\varphi_{\reg}$ is well defined and injective. In addition, surjectivity follows easily from 
formula \eqref{homeobasepoint} and the properties of Busemann functions (Lemma~\ref{lemma:propertiesbusemann}).
%
Again using formula~\eqref{homeobasepoint}, continuity of $\varphi_{\reg}^{-1}$ follows from construction. 

To prove continuity of $\varphi_{\reg}$, 
as ideal boundaries are metrizable, we use sequences. Let $\max\{\beta^{o}_{\xi_n}(z), \beta^{o'}_{\xi'_n}(z')-C_n\} $
be a sequence that converges to  $\max\{\beta^{o}_{\eta}(z), \beta^{o'}_{\eta'}(z')-C\} $.
The third coordinate of  $\varphi_{\reg}$ in \eqref{homeoreg} is clearly continuous, hence $C_n\to C$.
By compactness of $\partial_{\infty} X_i$, 
up to subsequence $\xi_n\to \xi_\infty$ and  $\xi'_n\to \xi'_\infty$.  By injectivity  of $\varphi_{\reg}$,  $\xi_\infty=\eta$ and $\xi'_\infty=\eta'$ and we get continuity.
\end{proof}
\begin{remark}\label{signos} 
Observe that 
a sequence $(x_n,y_n)$
converges to $(\xi, \xi', C)$ iff:
\begin{align*}
x_n &\rightarrow \xi\\
y_n &\rightarrow \xi'\\
d_1(x_n, o) - d_2(y_n, o')&\rightarrow C.
\end{align*}
\end{remark}
Let $\join(\partial_{\infty}X_1, \partial_{\infty} X_2)$ denote the topological join of $\partial_{\infty}X_1$ and $\partial_{\infty} X_2$.
Propositions~\ref{homeo singular} and ~\ref{homeo regular} can be improved as:

\begin{proposition}\label{joinprop} There is a natural homeomorphism
$$\partial_{\infty}^{\max}(X_1 \times X_2)\cong \join(\partial_{\infty}X_1, \partial_{\infty} X_2).$$ 
\end{proposition}

\begin{proof}
In view of Propositions~\ref{homeo singular} and \ref{homeo regular} and Remark~\ref{homeo regular with a base point},
we have to prove the following claim: 
for sequences $(\xi_n)$ in $\partial_{\infty} X_1$,  $(\xi_n')$ in $\partial_{\infty} X_2$, and $(C_n)$ in $\mathbb R$
we have $\xi_n\to  \xi$ and $C_n\to +\infty$ as $n\to\infty$ if and only if 
the function $$(z,z')\mapsto \max \{\beta_{\xi_n}^o(z), \beta_{\xi_n'}^{o'}(z')-C_n  \}   $$
converges to $(z,z')\mapsto \beta_{\xi}^o(z)$ uniformly on compact sets. Notice that we do no require convergence on $(\xi'_n)$.
We also need the symmetric claim when $C_n\to -\infty$, and after replacing $\max \{\beta_{\xi_n}^o(z), \beta_{\xi_n'}^{o'}(z')-C_n  \}$
by $\max \{\beta_{\xi_n}^o(z)+C_n, \beta_{\xi_n'}^{o'}(z')  \}$ (a different function in the same equivalence class) but  the proof is symmetric.

To prove the claim, assume first that $C_n\to + \infty$. Take as compact set the ball centered at $(o,o')$: $B(o, R)\times B(o', R)$.
As Busemann functions are 1-Lipschitz
and we chose normalizations so that $\beta_{\xi_n}^o(o)=  \beta_{\xi_n'}^{o'}(o')=0  $, for $C_n\geq  2 R$ we have 
$\max \{\beta_{\xi_n}^o(z), \beta_{\xi_n'}^{o'}(z')-C_n  \}=  \beta_{\xi_n}^o(z)$ for $(z,z')\in B(o, R)\times B(o', R)$.
Here uniform convergence of $\beta_{\xi_n}^o$ on compact subsets follows from the horosphere compactification of $X_1$.
Next  assume  $C_n\in [-R, R]$. Here 
$\max \{\beta_{\xi_n}^o(z), \beta_{\xi_n'}^{o'}(z')-C_n  \}$ has a converging subsequence 
to  $\max \{\beta_{\xi}^o(z), \beta_{\xi'}^{o'}(z')-C  \}$, uniformly on compact subsets.
Using that the Busemann functions have slope -1 in rays pointing to
the ideal point, we see that the limit $\max \{\beta_{\xi}^o(z), \beta_{\xi'}^{o'}(z')-C  \}$ cannot 
be expressed a Busemann function in any of the factors, $\beta_{\xi}^o$ or $\beta_{\xi'}^{o'}$. 
\end{proof}

\section{ Diagonal actions}

Let $\Gamma$ be an infinite quasi-convex group of   isometries of a proper $\mathrm{CAT}(-1)$ space $X$. In this section we consider the diagonal action of $\Gamma$ on $X\times X$:
\begin{align*}
\Gamma\times X\times X &\rightarrow X\times X\\
(\gamma, x, y) &\mapsto (\gamma x, \gamma y).
\end{align*}
The diagonal action extends continuously to the ideal boundary of the $\max$ compactification. The following is straightforward:

\begin{lemma}\label{diagonalonboundary} The diagonal action on the points of $\partial^{\max}_{\infty}(X\times X)$ is given by:
\begin{align*}
\gamma [\beta^o_{\xi}] &= [\beta^o_{\gamma\xi}]\\
\gamma [\max\{\beta^o_{\xi}, \beta_{\xi'}^{o} - C\}] & = [\max\{\beta^o_{\gamma\xi}, \beta_{\gamma\xi'}^{o} - C + \beta^o_{\xi'}(\gamma^{-1}o) - \beta_{\xi}^{o}(\gamma^{-1}o)\}].  
\end{align*}
\end{lemma}

\begin{remark} Under the identification in Remark \ref{homeo regular with a base point} the diagonal action maps a singular point $\xi$ to $\gamma\xi$,
and a regular point $(\xi,\xi', C)$, to $(\gamma\xi, \gamma\xi', C + \beta^o_{\xi}(\gamma^{-1}o) - \beta_{\xi'}^{o}(\gamma^{-1}o))$.
\end{remark}

 In this section we prove that there is an open subset 
$\Omega_{\Gamma}^{\max}\subset\partial_{\infty}^{\max}(X\times X)$ where the 
diagonal action of $\Gamma$ is properly discontinuous and cocompact. In
Subsection~\ref{projection} 
we prove that  the nearest point projection of $X\times X$ to
the diagonal $\Delta\subset X\times X$ extends continuously to
$\overline{X\times X}^{\max}$.
 In 
Subsection \ref{idealdomain} we use this projection to show that there 
exist such set $\Omega_{\Gamma}^{\max}\subset\partial_{\infty}^{\max}(X\times 
X)$. Moreover, we see that the action on the whole $X\times X \cup 
\Omega_{\Gamma}^{\max}$ is properly discontinuous and cocompact and that 
$\Omega_{\Gamma}^{\max}$ is the largest open set of the boundary that satisfies 
these conditions.  

\subsection{Extending the projection to the diagonal}
\label{projection}

The nearest point projection from $ X\times X$ to the diagonal for the $\max$ distance is given by the midpoint:
\begin{equation}
\label{eqn:projX}
\begin{array}{rcl}
\pi \colon X\times X &\rightarrow& \Delta\\
(x, y) &\mapsto& (m, m), 
\end{array}
\end{equation}
where $\Delta$ is the diagonal in $X\times X$:
$$ \Delta = \{(x, x) \;|\; x\in X\},$$
and $m$ is the midpoint of the geodesic segment joining $x$ and $y$.
By construction, $\pi$ is continuous and equivariant.

In this section we extend it continuously to a map 
$$
\tilde{\pi}\colon  \overline{X\times X}^{\max}  \rightarrow \overline{\Delta}^{\max}=\Delta\cup \Delta_\infty,
$$
where $ \overline{\Delta}^{\max}$ is the closure of $\Delta$ in $ \overline{X\times X}^{\max} $, and
 $$\Delta_{\infty} = \{(\xi, \xi)\; |\; \xi \in \partial_{\infty}X \},$$ 
 denotes the diagonal in $\partial_{\infty}X \times \partial_{\infty}X$. 
 For this purpose, we consider the decomposition
 $$
 \partial_{\infty}^{\max}(X\times X)= \partial_{\infty}^{\max}(X\times X)_{\sing} 
 \sqcup \varphi_{\reg}^{-1}(\Delta_{\infty}\times\mathbb{R}) \sqcup \Omega^{\max}
 $$ 
where
$$
\Omega^{\max} = \varphi_{\reg}^{-1}((\partial_{\infty}X\times\partial_{\infty}X\setminus\Delta_{\infty})\times\mathbb{R})\subset 
\partial_{\infty}^{\max}(X\times X)_{\reg}.
,$$
and and $\varphi_{\reg}$ is the homeomorphism in Proposition \ref{homeo regular}.
In  \cite{Tere} the projection is extended continuously to a map
$$
\Omega^{\max}\to \Delta
$$
Following \cite{Tere}, the extension uses  that $ \Omega^{\max} $ is naturally homeomorphic to the set $G$ of parameterized geodesics in $X$ (with the topology of uniform convergence on compact sets) through the map:
\begin{align*}
\varphi\colon  G &\longrightarrow \Omega^{\max}\\
 g &\mapsto \lim_{n\rightarrow \infty} (g(n), g(-n))
\end{align*}
Via this identification,  by \cite{Tere} the projection extends continuously to 
\begin{equation}
 \label{eqn:projG}
\begin{array}{rcl}
 G & \to & \Delta \\
 g & \mapsto & (g(0),g(0))
\end{array} 
\end{equation}
Thus it remains to extend it to $\varphi_{\reg}(\Delta_\infty\times\mathbb{R})$ and to $\partial_{\infty}^{\max}( X\times X)_{\sing}$.

\begin{definition} The extended projection 
$$
\tilde{\pi}\colon  \overline{X\times X}^{\max}  \rightarrow \overline{\Delta}^{\max}\cong\overline{X}
$$
is defined by \eqref{eqn:projX} and \eqref{eqn:projG}  on $X\times X\cup \Omega^{\max}$. 

On  $\varphi_{\reg}(\Delta_\infty\times\mathbb{R})$ it is the projection to
$\Delta_\infty$, 
and on  $\partial_{\infty}^{\max}( X\times X)_{\sing}\cong \partial_{\infty}X \sqcup \partial_{\infty}X$ it is the identification 
$ \partial_{\infty}X\cong \Delta_\infty$.
%
\end{definition}

\begin{remark}\label{identificationgeodesics} 
We have an equivariant homeomorphism $\varphi'=\varphi_{\reg}\circ \varphi$, given by:
\begin{align*}
\varphi'\colon  G &\rightarrow ((\partial_{\infty}X\times \partial_{\infty}X)\setminus{\Delta_{\infty}})\times\mathbb{R}\\
g &\mapsto (g(+\infty), g(-\infty), C_g),
\end{align*}
where:
$$ C_g = \lim_{n\rightarrow \infty} d(g(n), o) - d(g(-n), o) = \beta^{o}_{g(-\infty)}(g(0)) - \beta^{o}_{g(+\infty)}(g(0)).$$
Therefore, a geodesic $g$ corresponds to a point 
$$(g(+\infty), g(-\infty), C_g)\in ((\partial_{\infty}X\times \partial_{\infty}X)\setminus{\Delta_{\infty}})\times\mathbb{R} $$ 
which in its turn corresponds to the regular point: 
$$[\max\{\beta^o_{g(+\infty)}, \beta^o_{g(-\infty)} - C_g\}] \in \Omega^{\max}.$$
\end{remark}
 
%

For the continuity of
 $\tilde{\pi}$, we  use of  that the Gromov product extends continuously to the boundary of a proper $\mathrm{CAT}(-1)$ space 
 \cite[Proposition~3.4.2]{Buyalo-Schroeder}. The Gromov product of two points $x,y\in X$ with respect to a base point $o\in X$ is defined as:
$$(x|y)_o = \frac{1}{2}\left[ d(x,o) + d(y,o) - d(x,y) \right]$$
For $\xi, \xi'$ two points in the visual boundary of a proper $\mathrm{CAT}(-1)$ the Gromov product is defined as:
$$\lim_{i,j}(x_i|y_j)_o= (\xi|\xi')_o,$$
for any sequences $x_i\rightarrow \xi$, $y_j\rightarrow \xi'$. 

For $g$ a geodesic in $X$, the Gromov product of the ideal points $g(+\infty)$ and $g(-\infty)$ with respect to a base point $o$, can be written in terms of Busemann functions as:
$$(g(+\infty)|g(-\infty))_o= \frac{1}{2}\left[\beta^ {g(0)}_{g(+\infty)} (o) + \beta^{g(0)}_{g(-\infty)}(o)\right].$$
The Gromov product for ideal points satisfies:
 $$(\xi|\xi')_o= +\infty \text{ if and only if } \xi=\xi',$$
see \cite{Buyalo-Schroeder}. Similarly two sequences $x_i$, $y_j$ have the same limit iff:
$$(x_i|y_j)_o= +\infty.$$

\begin{theorem} The map $\tilde{\pi}\colon  \overline{X\times X}^{\max} \rightarrow \overline{X}$ is continuous and equivariant.
\end{theorem}
\begin{proof}
The equivariance  follows from naturality. 
For the continuity, 
we have also seen in \cite{Tere} that $\tilde{\pi}$ restricted to $X\times X \cup \Omega^{\max}$ is continuous, 
but it remains to be proved that it is continuous in the rest of the points. There are mainly two cases to check. First, we have to see that the image of a sequence of points $(x_n, y_n)$ in $X\times X$ that converges to an ideal point, either in the singular part or in the diagonal of the regular part of the boundary, converges to the image of this ideal point. This is case (I) in the proof. Second, we have to check that the image of a sequence of ideal points that converges to an ideal point either in the diagonal of the regular part or in the singular part of the boundary, converges to the image of the ideal point. This is case (II) in the proof. Along all the proof, $m_n$ will denote the midpoint of the segment joining $x_n$ and $y_n$.

{\bf Case (I)}.
Consider a sequence $(x_n, y_n)$ in $X\times X$ converging to an ideal point. We distinguish two subcases: the limit of the sequence is a singular point (a) or the limit is a point in the diagonal of the regular part (b).

 \emph{Subcase (a)}. Suppose, up to permuting factors, that the sequence converges to a singular point in the boundary of the first factor: $(x_n, y_n) \rightarrow [\beta_{\xi}^o]$. Therefore,  $x_n\rightarrow \xi$ and $d(x_n,o)-d(y_n,o)\rightarrow +\infty$. By the triangle inequality: $d(m_n, o) \geq d(m_n, y_n) - d(y_n, o)$, and using the definition of the Gromov product, we have:
\begin{eqnarray*}
(x_n|m_n)_o  \geq  \frac{1}{2}\left[d(x_n,o) -d(y_n,o)\right].
\end{eqnarray*}
Henceforth $(x_n|m_n)_o\rightarrow +\infty$, and by the properties of the Gromov product, $x_n$ and $m_n$ have the same limit. Therefore, $\tilde{\pi}((x_n, y_n)) = m_n \rightarrow \tilde{\pi}([\beta_{\xi}^o])=\xi$.

 \emph{Subcase (b)}. Now suppose that the sequence converges to a diagonal point in the regular part of the boundary: $(x_n, y_n) \rightarrow [\max\{\beta_{\xi}^o,\beta_{\xi}^o+ C\}]$. In this case $x_n\rightarrow \xi$, $y_n\rightarrow \xi$ and $d(x_n, o) -d(y_n,o) \rightarrow -C$. 

Using the definition of the Gromov product again and reorganizing terms, we have: 
\begin{equation}
\label{eq1}
2(x_n|m_n)_o= (x_n|y_n)_o + d(m_n, o) + \frac{1}{2}\left[d(x_n,o) - d(y_n,o)\right].
\end{equation}
\noindent Observe that $d(m_n, o)\geq 0$, $\frac{1}{2}\left[d(x_n,o) - d(y_n,o)\right]$ is 
uniformly bounded and $(x_n|y_n)_o\rightarrow\infty$, since both $x_n$ and $y_n$ converge to the same point. Then, from equation \eqref{eq1}, we deduce that $(x_n|m_n)_o \rightarrow +\infty$, which implies that $m_n\rightarrow \xi$, and $\tilde{\pi}((x_n, y_n)) = m_n \rightarrow \tilde{\pi}([\max\{\beta_{\xi}^o,\beta_{\xi}^o+ C\}])=\xi$.    

\smallskip
{\bf Case (II).}
Next we deal with a sequence of regular ideal points of the form $[\max\{\beta^{o}_{\xi_n}, \beta^{o}_{\xi_n'}+C_n\}]$ in $\partial_{\infty}^{\max}(X\times X)$ with limit a regular point in the diagonal, subcase (a), or a singular point, subcase (b). From now on, $g_n$ will denote the geodesic corresponding to a point 
$[\max\{\beta^{o}_{\xi_n}, \beta^{o}_{\xi_n'}+C_n\}]$ under the identification $\Omega^{\max} \cong G$.

 \emph{Subcase (a)}. Suppose that the sequence converges to a regular diagonal point: $[\max\{\beta^{o}_{\xi_n}, \beta^{o}_{\xi_n'}+C_n\}] \rightarrow [\max\{\beta^{o}_{\xi}, \beta^{o}_{\xi}+C\}].$ In this case $\xi_n\rightarrow \xi$, $\xi'_n\rightarrow \xi$ and $C_n\rightarrow C$.
Now, for each $n$, we consider a sequence of points $x_k$ in $X$ such that $x_k\rightarrow \xi_n$. Using the fact that the Gromov product extends continuously to the boundary of a $\mathrm{CAT}(-1)$ space, and the definition of Busemann function, we can write:
\begin{multline*}
 (g_n(0)|\xi_n)_o = \lim_{k\rightarrow \infty} (g_n(0)|x_k)_{o}= \lim_{k\rightarrow \infty} \frac{1}{2}\left[d(g_n(0), o) + d(x_k, o)\right. \\\left.- d(g_n(0), x_k)\right]= \frac{1}{2}\left[d(g_n(0), o) + \beta^{g_n(0)}_{\xi_n}(o)\right].
 \end{multline*}
Similarly, taking a sequence $y_k$ in $X$ for each $n$, with $y_k\rightarrow \xi'_n$:
$$(g_n(0)|\xi_n')_o = \frac{1}{2}\left[d(g_n(0), o) + \beta^{g_n(0)}_{\xi_n'}(o)\right].$$
Adding the two equalities above we obtain:
\begin{multline}
\label{eq2}
(g_n(0)|\xi_n)_o + (g_n(0)|\xi_n')_o = d(g_n(0), o)+\frac{1}{2}\beta^{g_n(0)}_{\xi_n}(o)+\frac{1}{2}\beta^{g_n(0)}_{\xi_n'}(o)\\= d(g_n(0), o) + (\xi_n|\xi'_n)_o. 
\end{multline}
By compactness of $\overline{X}$, after passing to a subsequence we may assume that 
$g_n(0) \rightarrow \eta\in\overline{X}$. Then, since $d(g_n(0), o)\geq 0$ and $(\xi_n|\xi'_n)_o\rightarrow (\xi|\xi)_o= +\infty$, by equality \eqref{eq2}
we have that $(\eta|\xi)_o=+\infty$. So $\xi=\eta$ and $g_n(0)\rightarrow\xi$. Therefore, every convergent subsequent of $g_n(0)$ converges to $\xi$ and then $g_n(0)$ converges to $\xi$. So $\tilde{\pi}([\max\{\beta^{o}_{\xi_n}, \beta^{o}_{\xi_n'}+C_n\}] ) = g_n(0) \rightarrow \tilde{\pi}([\max\{\beta_{\xi}^o,\beta_{\xi}^o+ C\}])=\xi$.

\emph{Subcase (b)}. Next we suppose that the sequence converges to a singular point, a Busemann function in the second factor (up to permutation of factors):
 $[\max\{\beta^{o}_{\xi_n}, \beta^{o}_{\xi_n'}+C_n\}] \rightarrow [\beta^{o}_{\xi'}]$ with $\xi_n\rightarrow \xi$, $\xi_n'\rightarrow \xi'$ and since $\beta^{o}_{\xi'}$ is a Busemann function in the second factor,  $C_n\rightarrow +\infty$. Observe that  $\xi$ and $\xi'$ might be equal. 

Similarly to the preceding case, for each $n$:
$$(g_n(0)|\xi'_n)_o =\frac{1}{2}\left[ d(g_n(0), o) - \beta^o_{\xi'_n}(g_n(0))\right], $$
$$(g_n(0)|\xi_n)_o = \frac{1}{2}\left[d(g_n(0), o) - \beta^o_{\xi_n}(g_n(0))\right], $$
and we can combine the two equalities to get:
\begin{eqnarray*}
(g_n(0)|\xi'_n)_o&=& \frac{1}{2}d(g_n(0), o) - \frac{1}{2}\beta^o_{\xi'_n}(g_n(0))\\ &=&(g_n(0)|\xi_n)_o + \frac{1}{2}\beta^o_{\xi_n}(g_n(0))- \frac{1}{2}\beta^o_{\xi'_n}(g_n(0))\\ &=& (g_n(0)|\xi_n)_o + \frac{1}{2}C_n.
\end{eqnarray*}
Here we have used that $C_n =\beta^o_{\xi_n}(g_n(0))- \beta^o_{\xi'_n}(g_n(0))$ 
by Remark \ref{identificationgeodesics}. Now, since $(g_n(0)|\xi_n)_o\geq 0$ and 
$C_n\rightarrow +\infty$ we have that $(g_n(0)|\xi'_n)_o\rightarrow +\infty$ and 
$g_n(0)\rightarrow \xi'$. So $\tilde{\pi}([\max\{\beta^{o}_{\xi_n}, 
\beta^{o}_{\xi_n'}+C_n\}] ) = g_n(0) \rightarrow 
\tilde{\pi}([\beta^{o}_{\xi'}])=\xi'$.
\end{proof}
\subsection{The ideal domain $\Omega^{\max}_{\Gamma}$}
\label{idealdomain}

Let $\Gamma$ be an infinite quasi-convex group of   isometries of $X$. We 
denote by $\Lambda_{\Gamma}$ its limit set, which  is the set of accumulation 
points of an orbit in $\partial_{\infty}X$, and by  $\Omega_{\Gamma}$,  its 
domain of discontinuity, which is the complement of the limit set in 
$\partial_{\infty} X$. The action of $\Gamma$ on $X\cup \Omega_{\Gamma}$ is 
properly discontinuous and cocompact \cite{Coornaert2,Swenson}. 
Next we show that the diagonal action of $\Gamma$ on the inverse image under the 
projection $\tilde{\pi}$ of $X\cup \Omega_{\Gamma}$ is also properly 
discontinuous and cocompact:

\begin{theorem}
\label{quasiconvexdomain}
Let $X$ be a proper  $\mathrm{CAT}(-1)$ space and let $\Gamma \subset 
\operatorname{Isom}(X)$ be an infinite quasi-convex group. The diagonal action 
of $\Gamma$ on $\tilde{\pi}^{-1}(X\cup \Omega_\Gamma)$ is properly discontinuous 
and cocompact.
\end{theorem}
\begin{proof}
Besides being  continuous and equivariant, $\tilde{\pi}\colon  \overline{X\times X}^{\max} \rightarrow \overline{X}$
 is proper, since it is a continuous map from a compact to a Hausdorff space. 
 
In \cite{Swenson} it is shown that for a Dirichlet domain $D\subset X$ its closure in $\overline{D}\subset\overline X$
is a compact set that satisfies:
\begin{enumerate}[label=(\roman*)]  
\item $\overline{D}\subset X\cup\Omega_\Gamma$, 
\item $\bigcup_{\gamma\in\Gamma} \gamma \overline{D}=X\cup\Omega_\Gamma$,
and 
\item for every compact $K\subset X\cup\Omega_\Gamma$, $|\{\gamma\in\Gamma\mid \gamma K\cap K\neq\emptyset\}|<\infty$. 
\end{enumerate}
Therefore,  $\tilde{\pi}^{-1}( \overline{D} )$ satisfies the corresponding properties for the action on $\tilde{\pi}^{-1}(X\cup \Omega_\Gamma)$,
that imply the theorem.
\end{proof}

Now, let $\Omega^{\max}_{\Gamma}$ be the intersection of $\tilde{\pi}^{-1}(X\cup \Omega_{\Gamma})$ with the ideal boundary of $\overline{X\times X}^{\max}$:
$$\Omega^{\max}_{\Gamma}= \tilde{\pi}^{-1}(X\cup \Omega_{\Gamma})\cap \partial_{\infty}^{\max}(X\times X),$$
and let $\Delta_{\Gamma}$ be the subset of the diagonal in $\partial_{\infty} X \times \partial_{\infty} X$ that corresponds to the limit points of the action of $\Gamma$ on $X$:
  $$\Delta_{\Gamma}= \{(\xi, \xi) \in \partial_{\infty} X \times  \partial_{\infty} X\text{ with }\xi\in \Lambda_{\Gamma}\}.$$

\begin{remark} Via the homeomorphism in Proposition \ref{joinprop} that identifies $\partial_{\infty}^{\max}(X\times X)$ with $\join(\partial_{\infty} X, \partial_{\infty} X)$:
\label{homeoomega}
$$\Omega^{\max}_{\Gamma}\cong((\partial_{\infty} X \times  \partial_{\infty} X)\setminus \Delta_{\Gamma})\times \mathbb{R} \cup \Omega^1_{\Gamma} \cup \Omega^2_{\Gamma}, $$
where $\Omega^i_{\Gamma}$ is $\Omega_{\Gamma}$ seen in the factor $\partial_{\infty}X_i$, for each $i=1,2$. Observe that when $\Gamma$ is a cocompact group $\Omega_{\Gamma}^{\max}$ is just the set $\Omega^{\max}$ of the previous subsection.
\end{remark}

In Proposition \ref{uniquedomain} we will show that $\Omega^{\max}_{\Gamma}$ is 
the largest open set of the boundary where the diagonal action is properly 
discontinuous. But first, let us study the limit set of this action on 
$\partial_{\infty}^{\max}(X \times X)$. Since $(X\times X, d_{\max})$ is no 
longer $\mathrm{CAT}(0)$, the set of accumulation points of each orbit depends 
on the orbit. We define the \emph{large limit set} of the diagonal action as the union 
of all the accumulation sets of each orbit on 
$\partial_{\infty}^{\max}(X\times X)$:
$$\Lambda = \displaystyle\bigcup_{(x, y) \in X\times X} \overline{\Gamma(x,y)} 
\cap \partial_{\infty}^{\max}(X \times X) $$ 
\begin{lemma}
\label{limit set} The large limit set $\Lambda$ of the diagonal action of $\Gamma$ on $\overline{X\times X}^{\max}$ seen under the homeomorphism $\varphi_{\reg}$ is
$$\Delta_{\Gamma}\times \mathbb{R}.$$
\end{lemma}
\begin{proof}
First, observe that the limit of any sequence $(\gamma_n x, \gamma_n y)$ that 
converges to the ideal boundary is contained in $\Delta_{\Gamma}\times 
\mathbb{R}$. Indeed, by the triangle inequality, $|d(\gamma_nx, o)- d(\gamma_n 
y,o)|\leq d(x,y)$, so the limit will be a regular point. Moreover if $\gamma_n 
x\rightarrow \xi$ then $\gamma_n y\rightarrow \xi$ since $d(\gamma_nx, \gamma_n 
y)= d(x,y)$, i.e.~the sequences $\gamma_n x$ and $\gamma_n y$ stay within a 
bounded distance. Therefore the limit point lies in $\Delta_{\Gamma}\times 
\mathbb{R}$. 

Moreover, any point $(\xi, \xi, C)$ for $\xi\in\Lambda_{\Gamma}$ belongs to the limit set. Take any sequence $\gamma_n$ such that $\gamma_no\rightarrow \xi$. Let $\xi'$ be an accumulation point for $\gamma^{-1}_n o$ and $x$,$y$ two points satisfying:
$$\beta^o_{\xi'}(y)-\beta^o_{\xi'}(x)=C.$$ 
Then $(\xi, \xi, C)$ is the limit of the sequence $(\gamma_nx, \gamma_ny)$. 
\end{proof}

\begin{remark} 
\label{remark:largelimitset}
The large limit set $\Lambda$ is not closed but observe that the complement of $\Omega^{\max}_{\Gamma}$ is $\overline{\Lambda}$, which is indeed closed.
\end{remark}
\begin{proposition}
\label{uniquedomain} $\Omega^{\max}_{\Gamma}$ is the largest open subset of $\partial_{\infty}^{\max}(X\times X)$  such that the action of $\Gamma$ on $X\times X\cup \Omega^{\max}_{\Gamma}$ is properly discontinuous.
\end{proposition}
\begin{proof} By Remark \ref{homeoomega}, any open set $A$ containing $\Omega^{\max}_{\Gamma}$, and strictly larger than $\Omega^{\max}_{\Gamma}$, contains points of $\Delta_{\Gamma}\times \mathbb{R}$. Then, by Lemma \ref{limit set}, $A$ contains points of the limit set so the action on $A$ cannot be properly discontinuous.
\end{proof}

\section{Examples}
In this section we consider some examples for the $\max$ compactification of 
diagonal actions. The first one is the diagonal action of a cocompact group of 
isometries of a riemannian manifold. The second one is the action of
convex cocompact Kleinian groups on 
$\mathbb{H}^n\times\mathbb{H}^n$. Finally we describe an example of the 
max compactification of a diagonal action on the product of two trees. This is 
also an example where the nearest point projection to the diagonal is not a 
fibration.
\subsection{Compact riemannian manifolds}
\label{maxcompactificationriemannian}
Recall that the fibre, restricted to $X\times X$, of $\tilde{\pi}$ over a point $(a,a)$ in the diagonal is the set of points $(x,y)$ such that $a$ is the midpoint of the segment joining $x$ and $y$. If $X=\mathbb{H}^2$, then the fibre over $(a,a)\in\Delta\subset \mathbb{H}^2\times\mathbb{H}^2$ is the set:
$$F_a = \{(x, s_ax)\; | \; x\in\mathbb{H}^2\} $$
where $s_ax$ is the symmetric point of $x$ with respect to $a$. Then $F_a\cong \mathbb{H}^2$ for any $a$, through the map $(x, s_ax)\mapsto x$, and all the fibres over the diagonal are homeomorphic to disks.
The boundary at infinity of the fibre over $(a,a)$ is the set of parameterized geodesics with $g(0)=a$, so:
$$\partial_{\infty}^{\max} F_a = \{ g\;|\; g(0)= a \}\cong (T_a\mathbb{H}^2)^1$$
For $S=\mathbb{H}^2/\Gamma$ a compact hyperbolic surface, the $\max$ compactification of $(\mathbb{H}^2\times\mathbb{H}^2)/\Gamma$ is the fibration by closed disks of $S=\mathbb{H}^2/\Gamma$, so:
$$\overline{(\mathbb{H}^2\times\mathbb{H}^2)/\Gamma}^{\max} \cong US,$$
where $US =\{ (x, v)\in TS \; | \; |v|\leq 1\}$.

The same is true for a Cartan-Hadamard manifold $X$ of dimension $n$ and 
sectional curvature 
$\leq -1$: $\partial_{\infty}F_x$ is identified with the unit tangent sphere at $x$, 
$(T_{x}X)^1 \cong S^{n-1}$, and the fibre over each point of the diagonal is a 
closed disk. If $M=X/\Gamma$ is a compact 
manifold, then the compactification of $(X\times X)/\Gamma$ with respect to the 
$\max$ metric is homeomorphic to the fibration by closed disks of $M=X/\Gamma$:
$$\overline{(X\times X)/\Gamma}^{\max} \cong UM,$$
where $UM =\{ (x, v)\in TM \; | \; |v|\leq 1\}$.

\subsection{Convex cocompact Kleinian groups}

Let $\Gamma <\operatorname{Isom}^+(\mathbb{H}^n)  $ be a discrete torsion free subgroup, that is convex cocompact.
Assume that $M=\mathbb{H}^n/\Gamma$ is not compact, then it has finitely many ends and its
compactification consists in adding a compact conformal $(n-1)$ manifold $N_i^{n-1}=\Omega_i/\Gamma_i$,
where $\Omega_i$ is a connected component of the discontinuity domain $\Omega=\partial_\infty \mathbb{H}^3
\setminus \Lambda_\Gamma$.

The compactification $\overline{(\mathbb{H}^n\times\mathbb{H}^n)/\Gamma}^{\max}$ is the union of 
the fibration by
compact balls on $M$ and a finite collection of products of the conformal ideal manifolds with intervals,  
$N_i^{n-1}\times \overline{\mathbb{R}}$, where $\overline{\mathbb{R}}=\mathbb{R}\cup\{-\infty,+\infty\}\cong [0,1]$.

To understand how these products $N_i^{n-1}\times [0,1]$ are attached, we consider a diverging geodesic 
ray in $r\colon [0,+\infty)\to M$, corresponding or a point in  $\partial_\infty 
M$.
For each $t\in [0,+\infty)$, the fibre $\pi^{-1} (r(t))$ is a compactified  hyperbolic space $\overline{\mathbb{H}^n}$,
we aim to understand how they fit with a compactified $\mathbb{R}$ when $t\to\infty$. Assume that 
$r$ is a ray in $\mathbb{H}^n$.  Points in  $\pi^{-1} (r(t))$ are of the form 
$(\exp_{r(t)}(  v), \exp_{r(t)}(-  v))$, for some  
$v\in T_{r(t)}\mathbb H^n$. To compare different fibers,  let $ V $ be a parallel vector field along $r$
that is unitary. Let $\theta\in [0,2\pi)$ be the angle between $r'(t)$ and $V(t)$, which is constant by parallelism. Then every point of 
$\pi^{-1} (r(t))$ is written as
$$
(\exp_{r(t)}( s\, V (t)), \exp_{r(t)}(-  s\, V (t)))
$$
for some $V $ as above and $s\in\mathbb{R}_{\geq 0}$.

\begin{proposition}\label{colapseSchottky}
Given  $V$ as above and $s\in\mathbb{R}_{\geq 0}$, 
$$
\lim_{t\to+\infty} (\exp_{r(t)}( s\, V (t)), \exp_{r(t)}(-  s\, V 
(t)))= \max\{ \beta_{r(+\infty)}^{r(0)}, \beta_{r(+\infty)}^{r(0)}- 2 d\}
$$
where $d\in\mathbb{R}$ is defined by the relation 
\begin{equation}
\label{eqn:projection}
\tanh d= \cos\theta \tanh s
 \end{equation}
and $\theta$ is the (constant) angle between $r'$ and $V$. 
\end{proposition}

The relation \eqref{eqn:projection}  means that $d$  is the signed distance from 
$r(t)$ to the orthogonal projection of $\exp_{r(t)}( s\, V (t))$ to 
the ray $r$, see Figure~\ref{fig:triang}.
This proposition explains how  the  $\overline{\mathbb{H}^n}$   in the fibre are attached 
to a segment in the limit: the whole $\mathbb{H}^n$ is projected orthogonally to the geodesic containing $r$,
and, by continuous extension,  $\overline{\mathbb{H}^n}$ is projected to $\overline{\mathbb{R}}=\mathbb{R}\cup\{-\infty,+\infty\}$.

\begin{figure}
\begin{center}
\begin{tikzpicture}[scale=.7]
\begin{scope}[rotate=30]
 \draw [thick] (-2,0) -- ( 5,0);
 \draw [thick] (0,0) -- ( 3,2);
 \draw [thin]  (3,0) --  (3,2) ;
\draw [thin] (1,0) arc [radius=1, start angle=0, end angle= 35];
 \end{scope}
 \node at (.5,-.5) {$r(t)$};
 \node at (1.5,.5) {$d$};
 \node at (4.5,2.2) {$r$};
 \node at (1,1) {$\theta$};
 \node at (.5,2) {$s$};
 \node at (2.5,3.6) {${\exp_{r(t)}( s\, V (t))}$};  
 \end{tikzpicture}
\end{center}
   \caption{
   The triangle for Equality~\eqref{eqn:projection}\label{fig:triang} in Proposition~\ref{colapseSchottky}}
\end{figure}
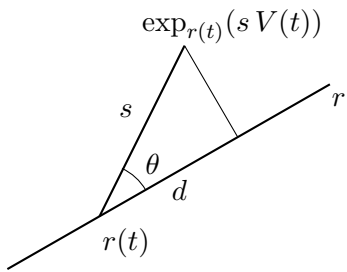

\begin{proof}
As the distance from $\exp_{r(t)}( \pm s\, V (t))$ to $r(t)$ is  $s$, we 
know that the limit is a maximum of Busemann functions centered at 
$r(+\infty)$. 
Therefore we only have to compute the limit 
$$
\lim_{t\to+\infty} d(\exp_{r(t)}( s\, V (t)), r(0)) - d( \exp_{r(t)}(-  
s\, V (t)), r(0) ).
$$
Set $d_{\pm}(t)=d(\exp_{r(t)}( \pm s\, V (t)), r(0))$. By the hyperbolic 
cosine formula:
\begin{align*}
 \cosh d_+(t) = \cosh s\cosh t - \cos\theta \sinh s\sinh t ,\\
 \cosh d_-(t) = \cosh s\cosh t + \cos\theta \sinh s\sinh t .
\end{align*}
Hence
$$
\lim_{t\to+\infty} e^{d_+(t)-d_-(t)}=\lim_{t\to+\infty}\frac{\cosh d_+(t) 
}{\cosh d_-(t)  }
=\frac{\cosh s-\cos \theta\sinh s}{\cosh s+\cos \theta\sinh s} 
$$
By taking logarithms:
$$
\lim_{t\to+\infty} {d_+(t)-d_-(t)}=\log \frac{1-\cos\theta\tanh 
s}{1+\cos\theta\tanh s}
=\log \frac{1-\tanh d}{1+\tanh d}
=-2d . \qedhere
$$
\end{proof}

\subsection{Constant valence trees}
\label{exampletrees}
Let $T$ be the tree of valence $4$ and edges of unit length, it is the Cayley graph of $\mathbb F_2$,
the free group on two generators.
Let  $\Delta\subset T\times T$ denote the diagonal. If we consider the nearest point projection 
to the diagonal, then the fibres over different points of $\Delta$ might be 
different. In fact, given $(a,a)\in \Delta$ there are three possible fibres, 
which are topologically not equivalent, depending on whether $a\in T$ is a 
vertex, a midpoint of an edge or a generic point in an edge.
\begin{theorem}
\label{thm:tree}
For $T$ the tree of valence four and edges of unit length, the 
fibres over $\Delta$ are of one of the following mutually exclusive types:
\begin{enumerate}
\item\label{generic}Generic fibres. For $a\in T$ a generic point in an edge, the 
fibre over $(a,a)$ is a tree of valence $4$. The metric in the fibre depends on 
the distance from $a$ to its 
nearest vertex in $T$, $L$ with $0<L<1/2$.
Along any path through $(a,a)$, the length of consecutive edges alternate 
between $2L$ and $L'= 1-2L$. The point $(a,a)$ is the midpoint of an edge of 
length $2L$.
\begin{figure}[h!]
\centering
\includegraphics[scale=0.2, angle= 90]{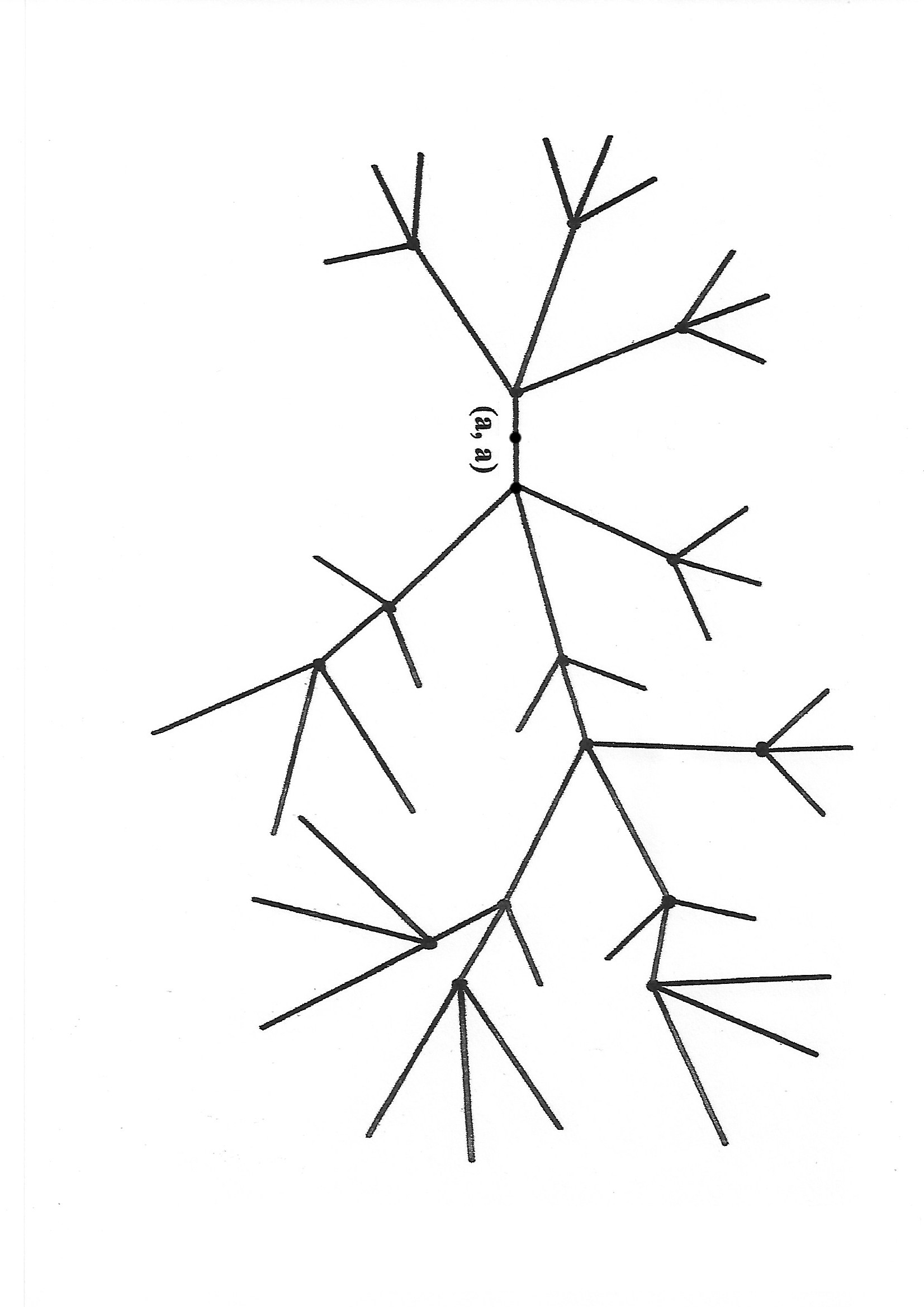}
\caption{Case 1}\label{arbol1}
\end{figure}

\item\label{midpoint}Midpoint fibres. For $a\in T$ a midpoint of an edge, the 
fibre over $(a,a)$ is a tree of constant valence $10$. All edges have length $1$ 
and the point $(a,a)$ is the midpoint of an edge. This is the limit of the 
previous case when $L\rightarrow 1/2$.
\begin{figure}[h!]
\centering
\includegraphics[scale=0.25, angle =90]{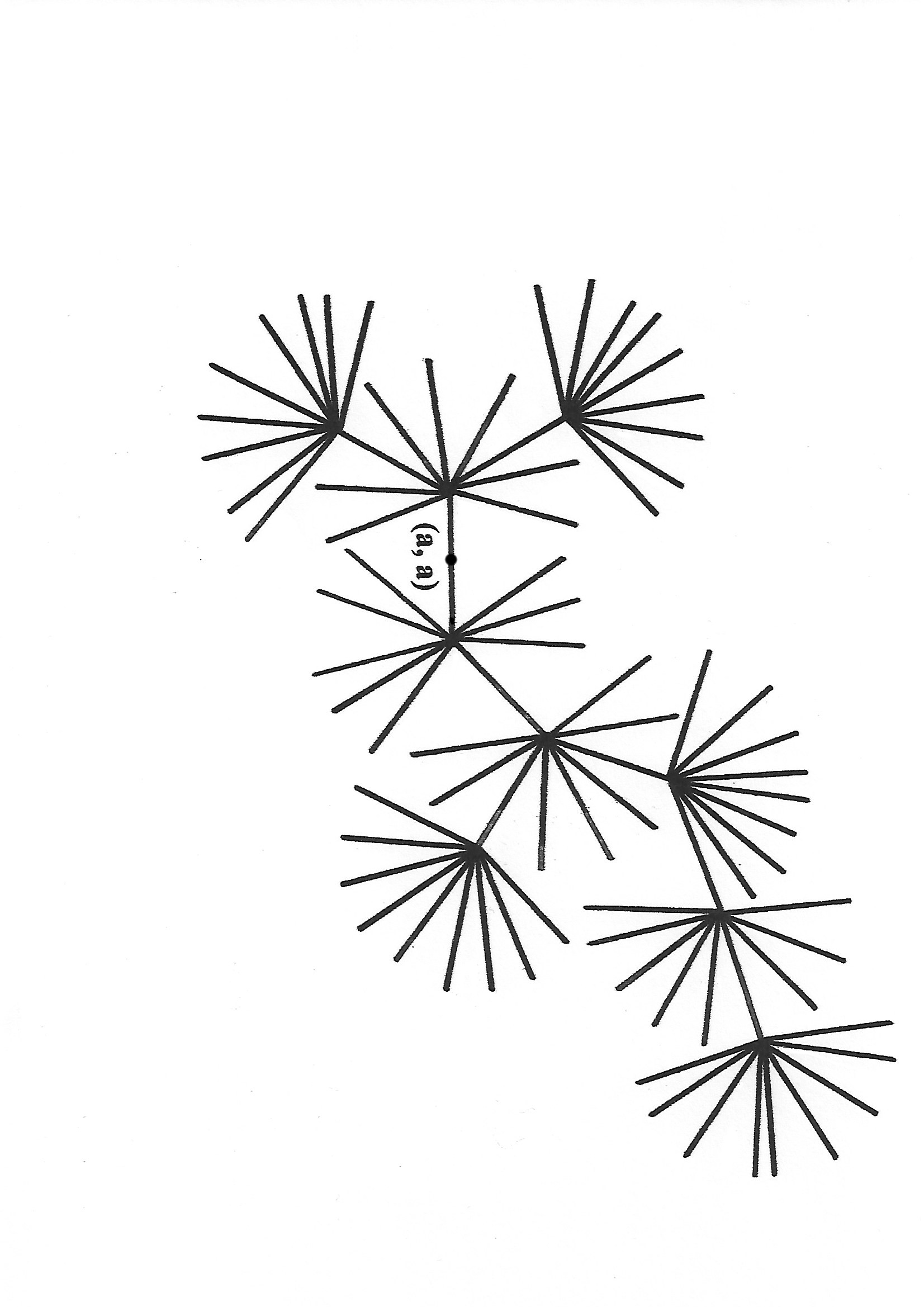}
\caption{Case~2}\label{arbol2}
\end{figure}

\item\label{vertex}Vertex fibres. For $a \in T$ a vertex, the fibre over $(a,a)$ 
is a tree of constant valence $10$, except in the point $(a,a)$, which is a 
valence $12$ vertex. All edges have length $1$.
This is the limit of the generic case when $L\rightarrow 0$, taking into account 
that there are four fibres approaching to the base point, one for each edge in 
$T$ incident to $a$.
\begin{figure}[h!]
\centering
\includegraphics[scale=0.2, angle= 90]{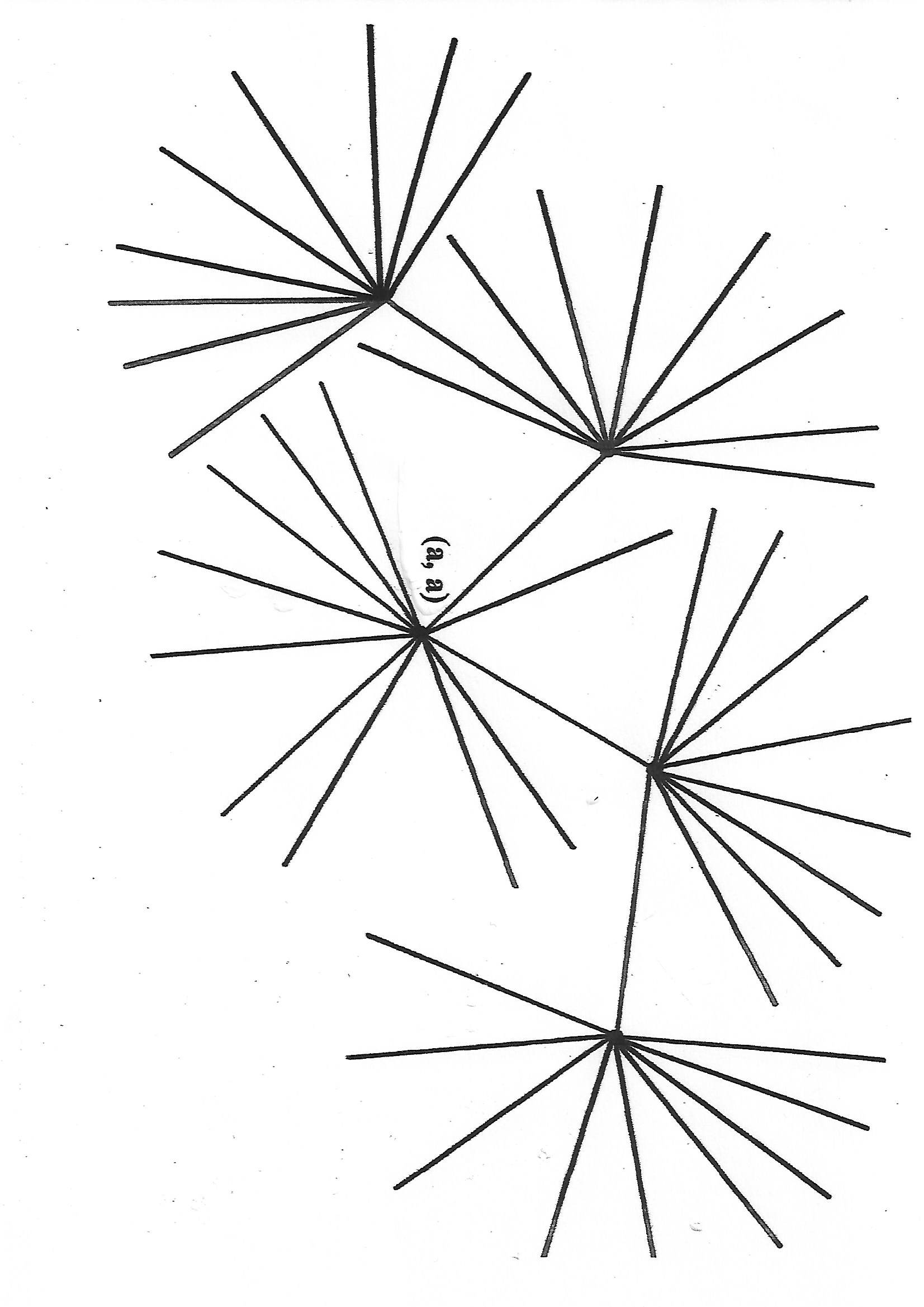}
\caption{Case 3}\label{arbol3}
\end{figure}
\end{enumerate}
\end{theorem}

\begin{proof}


The fibre of $(a,a)$ is the set of pairs 
$(x,y)\in T\times T$ so that the midpoint of the segment $\overline{xy}$ is $a$. To reach all such a pairs, we start from $a$ and 
consider pairs of paths obtained by moving at speed one along $T$ and pointing away from $a$.
 The first requirement is that
we start in different directions, ie for small times the points $(x, y)$ of the pair are already different. When one of the points reach a vertex, we consider all
possible continuations along different edges.

This construction provides  a graph structure to the fibre: when $x$ and $y$ move along an edge this yields an edge of the fibre, when at least one of them reaches a
vertex of $T$, this yields a vertex of the fibre. This fibre is in fact a tree, because we can orient each edge of the initial tree $T$
so that it points in the direction opposite $a$ (when $a$ lies in the interior of an edge, we split this edge along $a$). 
This yields an orientation of the edges of the fibre, so that each edge points away from $(a,a)$. In addition,
at every vertex only edge points to this vertex, the other edges point away,  hence it is a tree. 

We describe the tree for the fibre of a  vertex in $T$, case 3, the other two cases are a  
follow from similar arguments. Let $a$ be a vertex and 
denote by $(x,y)$ a point in the fibre over $(a,a)$. For each $x$ in a edge 
incident to $a$, there are three possibilities for $y$, such that $a$ is the 
midpoint of $x$ and $y$, one from each of the three remaining edges. In total 
there are $4$ edges incident to $a$, so there are $4\cdot 3 =12$ edges incident to 
$(a,a)$. Next we follow $x$ and $y$ along two edges, always satisfying 
$d(x,a)= d(y,a)$, until both $x$ and $y$ are two vertices $v$ and $v'$. Then, 
for $x$ there are three new possibilities, one for each new edge incident to 
$v$, and if we follow the path along one of the edges there are three 
possibilities for $y$, such that $d(x,a)= d(y,a)$, one for each new edge incident to 
the vertex $v'$. So there are $3\cdot 3 + 1= 10$ edges incident to $(v, 
v')$. This pattern repeats, given rise to a tree with valence $10$ in all 
vertices except the base point. The distance between two consecutive vertices, 
for instance $(a,a)$ and $(v,v')$, is $d_{\max}((a,a), (v, v')) = \max\{d(a,v), 
d(a,v')\} = d(a,v) =1$, so the edges have length $1$.  
\end{proof}
\begin{remark}
Let $T$ be again the tree of valence $4$, also the Cayley graph of $\mathbb{F}_2$.
The projection  $T\times T\to \Delta\cong T$ induces a map from $(T\times T)/\mathbb{F}_2 $ to 
$T/\mathbb{F}_2$, which is a wedge of two circles. The fibres are trees, as in 
Theorem~\ref{thm:tree}, and the tree depends on the point on the wedge $T/\mathbb{F}_2$: Case~3
for the vertex of the wedge, Case 2 for the midpoints of the edges, and Case~1
for the remaining (generic) pints.
The ideal boundary of each fibre in the max compactification is its 
boundary at infinity as a tree, which is a Cantor set in any case.
\end{remark}

The previous considerations of course apply to free groups of rank $n$ and their Cayley graphs.


\section{Product actions}
Let $(X_1, d_1)$ and $(X_2, d_2)$ be two proper $\mathrm{CAT}(-1)$ spaces and 
let $\Gamma$ be an infinite hyperbolic group with $\rho_1\colon \Gamma 
\rightarrow \operatorname{Isom}(X_1)$, $\rho_2\colon \Gamma \rightarrow \operatorname{Isom}(X_2)$ two \emph{quasi-convex 
representations} (by $\rho_i$ \emph{quasi-convex} we mean that it has finite kernel and that $\rho_i(\Gamma)$ is a discrete quasiconvex group). 
In this section we will study under what conditions the large 
limit set $\Lambda_{\rho_1\times\rho_2} $ of the product action
\begin{align*}
\Gamma\times X_1\times X_2 &\rightarrow X_1\times X_2\\
(\gamma, x, y) &\mapsto (\rho_1(\gamma) x, \rho_2(\gamma) y)
\end{align*}
lies inside the regular part of the boundary. We will see that asking the large limit set to lie 
in the regular part of the boundary is in fact a very restrictive condition, which is related to the marked length spectrum conjecture. 
Indeed, in Section \ref{regularlimitsets} we will prove the following proposition:

\begin{proposition}\label{equivalences} Let $X_1$, $X_2$ be proper $\mathrm{CAT}(-1)$ spaces and $\overline{X_1\times X_2}^{\max}$ the horofunction compactification with respect to $d_{\max}$. Let $\Gamma$ be an infinite hyperbolic group and $\rho_1\colon  \Gamma \rightarrow \operatorname{Isom}(X_1)$, $\rho_2\colon  \Gamma \rightarrow \operatorname{Isom}(X_2)$ two quasi-convex representations. The following are equivalent:
\begin{enumerate}[label=(\alph*)]
\item\label{limitset} $\Lambda_{\rho_1\times\rho_2}\subset \partial_{\infty}^{\max}(X_1\times X_2)_{\reg}$ .
\item\label{ce}$\rho_1 \simeq_{C.E.} \rho_2$.
\item \label{samemls}$\tau(\rho_1(\gamma)) = \tau(\rho_2(\gamma))$ for all $\gamma\in\Gamma$.
\end{enumerate}
\end{proposition}
Here $\tau$ denotes the translation length of an isometry, see Definition~\ref{coarselyequal} below,
so condition \ref{samemls} means that the two representations have the same 
translation lengths. Condition \ref{ce} in Theorem \ref{equivalences} stands for 
the two representations being coarsely equivalent:
\begin{definition}\label{coarselyequal} The   representations $\rho_1$, $\rho_2$ 
are said to be coarsely equivalent, written $\rho_1 \simeq_{C.E.} \rho_2$, if 
there exists $C>0$ such that:
$$ |d_1(\rho_1(\gamma)o, \rho_1(\gamma')o) - d_2(\rho_2(\gamma)o', 
\rho_2(\gamma')o')| \leq C$$
for some $o\in X_1$, $o'\in X_2$ and for all $\gamma, \gamma'\in \Gamma$.
\end{definition} 
If the representations are coarsely equivalent for some base-points $o\in X_1$, $o'\in X_2$ then they are coarsely equivalent for any 
choice of base-points in $X_1$ and $X_2$. 
We will see later in Lemma \ref{independenceoforbit} how the bound depends on  the choice of base-points.
\begin{definition}\label{def:almostisometry}
 A $K$-almost-isometry between two metric spaces is a map $f\colon X_1\to X_2$ such that 
 $$
\vert d(x,y)-d(f(x),f(y))\vert\leq K,\quad \forall x, y\in K,
 $$
 and $X_2$ lies in the $K$-neighborhood of $f(X_1)$.
\end{definition}

If $\rho_1$ and $\rho_2$ are coarsely equivalent and cocompact, then the spaces 
$X_1$ and $X_2$ are equivariantly almost-isometric. 
In this case it is possible to find a subset $\Omega^{\max}_{\Gamma}$ of 
$\partial_{\infty}^{\max}(X_1\times X_2)$ where the product action is properly 
discontinuous and cocompact.  In Section \ref{compactificationproduct} we use 
the existence of this almost-isometry between $X_1$ and $X_2$ and its extension 
to the ideal boundaries of the spaces to prove the following theorem:

\begin{theorem}
\label{Thm:Omega2}
Let $X_1$, $X_2$ be proper $\mathrm{CAT}(-1)$ spaces and 
$\overline{X_1\times X_2}^{\max}$ the horofunction compactification with respect 
to $d_{\max}$. Let $\Gamma$ be a hyperbolic group and $\rho_1\colon  \Gamma 
\rightarrow \operatorname{Isom}(X_1)$, $\rho_2\colon  \Gamma \rightarrow \operatorname{Isom}(X_2)$ two cocompact discrete
representations with finite kernel. If $\rho_1$ and $\rho_2$ are coarsely equivalent, then there 
exists a subset $\Omega^{\max}_{\Gamma}\subset 
\partial_{\infty}^{\max}(X_1\times X_2)$ where the product action of $\Gamma$ is 
properly discontinuous and cocompact.
\end{theorem}

\subsection{Regular limit sets and coarsely equivalent representations}
\label{regularlimitsets}
Let $\Gamma$ be a group acting on a space $X$ with two metrics $d_1$ and $d_2$ that are $\Gamma$-invariant.
\begin{definition}Two metrics $d_1$ and $d_2$  on a space $X$ are coarsely equivalent if there exists a constant $C\geq 0$ such that for all $x$, $y\in X$,
$$|d_1(x,y) - d_2(x,y)|\leq C.$$
\end{definition}
\begin{definition}\label{samemkls}
 The metrics $d_1$, $d_2$ have the same marked length spectrum with respect to the action of $\Gamma$ if $\tau_1(\gamma) = \tau_2(\gamma)$ for all $\gamma\in \Gamma$, where $\tau_i(\gamma)$ is the translation length of $\gamma$ for $d_i$ defined by 
$$\tau_i(\gamma) = \lim_{n\rightarrow\infty} \frac{d_i(x, \gamma^n(x))}{n}$$
for any $x\in X$.
\end{definition}
The equivalence of both definitions for hyperbolic groups follows from results of Furman \cite{Furman} and Krat \cite{Krat}. 
\begin{theorem}\label{equivalencescoarse}(Furman, Krat) Let $\Gamma$ be a 
hyperbolic group acting on itself with two left invariant metrics $d_1$, $d_2$ 
which are quasi-isometric to a word metric by the identity map. Then $d_1$ and $d_2$ are coarsely 
equivalent if and only if $d_1$ and $d_2$ have the same marked length spectrum.
\end{theorem}
We are interested in a hyperbolic group acting on two proper $\text{CAT(-1)}$ spaces $(X_1, d_1)$, $(X_2, d_2)$ via quasi-convex representations $\rho_1$, $\rho_2$ into their respective groups of isometries. A hyperbolic group  $\Gamma$ acts on itself by left translations. Moreover, fixing $o_i\in X_i$, each representation $\rho_i$ induces an orbit map from $\Gamma$ to the target space $X_i$ for $i=1,2$:
\begin{align*}
O_i\colon \Gamma &\rightarrow X_i\\
\gamma &\mapsto \rho_i(\gamma)o_i
\end{align*}
These orbit maps  induce left invariant metrics $d_{\Gamma_i}$ in $\Gamma$ by:
$$d_{\Gamma_i}(\gamma, \gamma') = d_{i}(\rho_i(\gamma)o_i, \rho_i(\gamma')o_i)$$
so that for $i=1,2$, $(\Gamma, d_{\Gamma_i})$ are $\Gamma$ invariant metric 
spaces. Moreover, since $(X_i, d_i)$ are proper $\mathrm{CAT}(-1)$ spaces and 
the representations are quasi-convex, these metrics are quasi-isometric to a word 
metric by the identity map, see  \cite{Bourdon}. 
\begin{remark}
\label{ceandrepresentations} The metrics $d_{\Gamma_1}$ and $d_{\Gamma_2}$
are coarsely equivalent if and only if $\rho_1\simeq_{C.E.} \rho_2$.
Indeed, in both cases the condition to be satisfied is that there exists a 
constant $C$ such that $$ |d_1(\rho_1(\gamma)o_1, \rho_1(\gamma')o_1) - 
d_2(\rho_2(\gamma)o_2, \rho_2(\gamma')o_2)| \leq C$$
for some $o_1\in X_1$, $o_2\in X_2$ and for all $\gamma, \gamma'\in \Gamma$.
\end{remark}
%

Using  Remark~\ref{ceandrepresentations}, Theorem~\ref{equivalencescoarse} yields:
\begin{proposition}
\label{ceandtranslationlength} Let $X_1$, $X_2$ be proper $\mathrm{CAT}(-1)$ spaces.
Let $\Gamma$ be an infinite hyperbolic group and $\rho_1\colon  \Gamma 
\rightarrow \operatorname{Isom}(X_1)$, $\rho_2\colon  \Gamma \rightarrow 
\operatorname{Isom}(X_2)$ two quasi-convex representations. Then:
$$\rho_1 \simeq_{C.E.} \rho_2 \Leftrightarrow \tau(\rho_1(\gamma))= 
\tau(\rho_2(\gamma))$$
for all $\gamma\in \Gamma$.
\end{proposition}
%

Now we will work towards the equivalence $(a)\Leftrightarrow(b)$ in Proposition~\ref{equivalences}. Recall that the \emph{large limit set} of 
the product action ${\rho_1\times\rho_2}$ is
 the union of the accumulation sets of all the orbits on $\partial_{\infty}^{\max}(X_1\times X_2)$:
$$\Lambda_{\rho_1\times\rho_2} = \displaystyle \bigcup_{(x,y)\in X_1\times X_2} \overline{ \{ (\rho_1(\gamma)x,\rho_2(\gamma)y)\mid\gamma\in \Gamma \} }
\cap \partial_{\infty}^{\max}(X_1\times X_2).$$
Recall also that the regular part of the ideal boundary can be identified with the product of the ideal boundaries of each factor and $\mathbb{R}$:
$$\partial_{\infty}(X_1\times X_2)_{\reg} \cong \partial X_1 \times \partial X_2 \times \mathbb{R}.$$
Fixing a base point $(o,o')\in X_1\times X_2$, a sequence $(x_n, y_n)\subset X_1\times X_2$ converges to a point $(\xi, \xi', C)$ in the regular part if:
$$x_n \rightarrow \xi\in \partial X_1,$$
$$y_n \rightarrow \xi' \in \partial X_2,\text{ and }$$
$$d_1(x_n, o) - d_2(y_n,o')\rightarrow C\in\mathbb{R}.$$

\begin{proposition}
If $\Lambda_{\rho_1\times\rho_2}\subset\partial_{\infty}^{\max}(X_1\times X_2)_{\reg}$ then $\rho_1\simeq_{C.E.}\rho_2$.
\end{proposition}
\begin{proof}
Suppose that $\rho_1$ and $\rho_2$ are not coarsely equivalent. This means that there is a sequence $\gamma_n$ in $\Gamma$
such that 
$|d_1(o, \rho_1(\gamma_n)o)-d_2(o', \rho_2(\gamma_n)o')|$ is unbounded. By definition of singular points, this means that 
$(\rho_1(\gamma_n)o, \rho_2(\gamma_n)o')$ accumulates in the singular part, which is a contradiction.
\end{proof}

For the implication $(b)\Rightarrow (a)$ we need a couple of lemmas. The first one is a direct consequence of the triangle inequality:
\begin{lemma}\label{triangleinequality} Let $x$, $y$, $z$ and $t$ be four points in a metric space $(X,d)$. Then:
$$|d(x,y) - d(z,t)| \leq d(x,z) + d(y,t).$$
\end{lemma}

\begin{lemma}
\label{independenceoforbit}
If $\rho_1\simeq_{C.E.}\rho_2$ then for any $x\in X_1$, $y\in X_2$:
$$|d_1(\rho_1(\gamma)x, \rho_1(\gamma') x) - d_2(\rho_2(\gamma)y, \rho_2(\gamma')y)| \leq C + 2(d_1(x,o) + d_2(y,o'))$$
for all $\gamma$, $\gamma'$ in $\Gamma$ and for a $C$ depending only on $o$ and $o'$.
\end{lemma}
\begin{proof}
We add and subtract $d_1(\rho_1(\gamma)o, \rho_1(\gamma')o)$ and $d_2(\rho_2(\gamma)o', \rho_2(\gamma') o')$ and apply the triangle inequality:
\begin{multline*}
|d_1(\rho_1(\gamma)x, \rho_1(\gamma') x) - d_2(\rho_2(\gamma)y, \rho_2(\gamma')y)|  \\
\leq  |d_1(\rho_1(\gamma)x, \rho_1(\gamma') x) - d_1(\rho_1(\gamma)o, \rho_1(\gamma')o)| \\
\quad \quad \quad \quad  + |d_1( \rho_1(\gamma)o , \rho_1(\gamma') o )-  d_2( \rho_2(\gamma) o', \rho_2(\gamma') o')| \\
  + |d_2(\rho_2(\gamma)o', \rho_2(\gamma') o')- d_2(\rho_2(\gamma)y, \rho_2(\gamma')y)|
\end{multline*}
Next we find a bound for each summand of the right-hand side. By Lemma~\ref{triangleinequality}:
\begin{multline*}
|d_1(\rho_1(\gamma)x, \rho_1(\gamma') x) - d_1(\rho_1(\gamma)o, \rho_1(\gamma')o)| \\
\leq d_1(\rho_1(\gamma)x,\rho_1(\gamma)o )+d_1(\rho_1(\gamma') x, \rho_1(\gamma')o) = 2d_1(x, o)
\end{multline*}
and
\begin{multline*}
|d_2(\rho_2(\gamma)o', \rho_2(\gamma')  o')-d_2(\rho_2(\gamma)y, \rho_2(\gamma')y)| \\
\leq  d_2(\rho_2(\gamma)o',\rho_2(\gamma)y) )+d_2(\rho_2(\gamma') o',\rho_2(\gamma')y)= 2d_2(y, o') .
\end{multline*}
In addition, by assumption:
$$|d_1( \rho_1(\gamma)o , \rho_1(\gamma') o )-  d_2( \rho_2(\gamma) o', \rho_2(\gamma') o')| \leq C,$$
so the result follows.
\end{proof}
\begin{remark} Observe that Lemma \ref{independenceoforbit} implies that the definition of coarse equivalence does not depend on the orbit.
\end{remark}



\begin{proposition}
If $\rho_1\simeq_{C.E.}\rho_2$ then $\Lambda_{\rho_1\times\rho_2}\subset\partial_{\infty}^{\max}(X_1\times X_2)_{\reg}$.
\end{proposition}
\begin{proof}
We want to show that the sequences of the form $(\rho_1(\gamma_n)x, \rho_2(\gamma_n)y)$ accumulate
in the regular part; 
equivalently $|d_1(\rho_1(\gamma_n)x, o) - d_2(\rho_2(\gamma_n)y, o')|$ is 
bounded, so every limit point of the sequence is in the regular part. 
Applying Lemma \ref{independenceoforbit} with $x=o$, $y=o'$, $\gamma=\gamma_n$ and $\gamma '= \mathrm{Id}$ we get that $|d_1(\rho_1(\gamma_n)o, o) - d_2(\rho_2(\gamma_n)o', o')| \leq C$. 
Then, it follows from the triangle inequality that $|d_1(\rho_1(\gamma_n)x, o) - d_2(\rho_2(\gamma_n)y, o')| \leq d_1(x,o) + d_2(y, o') + C$.
\end{proof}

\subsection{Compactification of product actions}
\label{compactificationproduct}
In this section we consider $\rho_1\colon  \Gamma\rightarrow 
\operatorname{Isom}(X_1)$, $\rho_2\colon  \Gamma\rightarrow 
\operatorname{Isom}(X_2)$ two discrete 
cocompact coarsely equivalent representations. We will show that, as in the 
diagonal case, there exists an open subset $\Omega\subset \partial^{\max}_\infty(X_1\times X_2)$ 
of the ideal 
boundary such that the product action on $X_1\times X_2\cup 
\Omega$ 
 is properly discontinuous and cocompact.

\begin{lemma}
\label{lemma:ai}
 If $\rho_1\colon  \Gamma\rightarrow \operatorname{Isom}(X_1)$, $\rho_2\colon  \Gamma\rightarrow \operatorname{Isom}(X_2)$ are coarsely equivalent cocompact representations, then there exists an equivariant almost-isometry $f\colon X_1\rightarrow X_2$.
\end{lemma}
\begin{proof} Since the action is cocompact on both spaces $X_1$ and $X_2$, each 
of these spaces is equivariantly almost-isometric to any orbit of $\Gamma$. The 
condition of coarse equivalence implies that the orbits of $\Gamma$ in $X_1$ are 
equivariantly almost-isometric to the orbits of $\Gamma$ in $X_2$.  
\end{proof}

\begin{remark}\label{almostorbits} 
The almost-isometry $f$ is not unique, because of the different  choices of orbits and choices of almost-isometries between the space $X_i$ and the orbit.
\end{remark}

To find $\Omega\subset \partial^{\max}_\infty(X_1\times X_2)$
such that $\Gamma$ acts properly discontinuously and cocompactly on
$X_1\times X_2\cup \Omega$, the basic idea is to use the map
$$
\mathrm{Id}\times f\colon X_1\times X_1\to X_1\times X_2
$$
(where $f\colon X_1\to X_2$ is the almost-isometry of Lemma~\ref{lemma:ai})
to translate the properties of the  
 diagonal action $\rho_1\times\rho_1$ on $ X_1\times X_1$
to the product action $\rho_1\times\rho_2$ on $X_1\times X_2$.

The  almost-isometry $f$  of Lemma~\ref{lemma:ai}  has an almost-inverse 
$f^{-1}\colon  X_2\rightarrow X_1$ such that: $$d_1(f^{-1}(f(x_1)), x_1)\leq K 
\qquad \text{ and }\qquad  d_2(f(f^{-1}(x_2)), x_2) \leq K,$$
for all $x_1\in X_1$ and $x_2\in X_2$.
Let $f$ be one of the almost-isometries in Remark \ref{almostorbits}. Since 
quasi-isometries between $\mathrm{CAT}(-1)$ spaces extend to homeomorphisms of 
the boundaries, $f$ extends to an equivariant homeomorphism: 
$$
f_\infty\colon 
\partial_{\infty}X_1 \rightarrow \partial_{\infty}X_2,
$$
whose inverse is the extension of the almost-isometry $f^{-1}$. 
\begin{remark} All the choices of almost-isometries in Remark \ref{almostorbits} 
extend to the same map $f_\infty\colon \partial_{\infty}X_1 \rightarrow 
\partial_{\infty}X_2$, since for any $x_i,x_i'\in X_i$, if $\rho_i(\gamma_n)x_i 
\rightarrow \xi$, then $\rho_i(\gamma_n)x_i' \rightarrow \xi$, for $i=1,2$.
\end{remark}

For $i=1,2$, let 
$$
\begin{array}{rcl}
 \varphi_i\colon (\partial_{\infty}^{\max}(X_1\times X_ i))_{\reg}
& \to & \partial_{\infty}X_1 \times \partial_{\infty}X_i\times \mathbb{R}  \\
z\qquad & \mapsto & (\xi_i(z), \eta_i(z), h_i(z) )
\end{array}
$$
be the homeomorphism of Proposition~\ref{homeo regular}.
Choose $o\in X_1$ and $f(o)\in X_2$ as base points to compute $h_1$ and $h_2$
as in Proposition~\ref{homeo regular}:
\begin{align*}
 h_1([(x,x')\mapsto \max\{  \beta(x),\beta'( x')   \}] & = \beta(o)-\beta'(o) , \\
 h_2([(x,x')\mapsto \max\{  \beta(x),\beta''(x'' )   \}] & = \beta(o)-\beta''(f(o))),
\end{align*}
where $\beta$ and $\beta'$ are Busemann functions on $X_1$ and $\beta''$ on $X_2$.

The domain of discontinuity of the diagonal action is
$$
\Omega_1= \varphi_1^{-1} ((\partial_\infty X_1\times \partial_\infty X_1\setminus\Delta_\infty)\times \mathbb R ),
$$
where $\Delta_\infty$ denotes the diagonal of $  \partial_\infty X_1$. 
For the action on $X_1\times X_2$ define
$$
\Omega_2= \varphi_2^{-1} ((\partial_\infty X_1\times \partial_\infty X_2\setminus\Delta_{f_\infty})\times \mathbb R ).
$$
where $\Delta_{f_\infty}$ is graph of $f_\infty$:
$$\Delta_{f_\infty} =\{(\xi, \eta)\in \partial_{\infty}X_1 \times 
\partial_{\infty}X_2 \;|\; \eta=f_\infty(\xi)\}.$$

By Remark~\ref{remark:largelimitset}, $\Omega_1=\partial_\infty^{\max} ( X_1\times X_1)\setminus \overline \Lambda_1$, where $\Lambda_1$ denotes the large limit set of the diagonal action. 
For $\Omega_2$ we also have:

\begin{remark}
\label{remark_largelimitset2}
 $\Omega_2= \partial_\infty^{\max}( X_1\times X_2)\setminus \overline \Lambda_2$, where $\Lambda_2$ denotes the large limit set of the $(\rho_1\times\rho_2)$-action.
\end{remark}

As Remark~\ref{remark:largelimitset}, this remarks follows from the fact that $\Lambda_2\cong \Delta_{f_\infty}\times\mathbb{R}$, that can be proved with  the same
proof as  Lemma~\ref{limit set}.

We will prove that $\Omega_2$ is the set $\Omega^{\max}_\Gamma$ of Theorem~\ref{Thm:Omega2}
For this purpose we consider a map
$
F\colon \Omega_1\to \Omega_2
$ defined as follows.
Every $z\in \Omega_1$ can be written as 
$$
z=\lim_{n\to +\infty} (g(-n), g(n)) 
$$
for a unique geodesic $g$ in $X_1$ (this construction yields a homeomorphism between the set of bi-infinite geodesics in $X_1$ and $\Omega_1$).
Next, if $\phi_1(z)=(\xi(z),\eta(z), h_1(z))$, then define $F(z)$ by
$$
\phi_2(F(z))=(\xi(z),f_\infty(\eta(z)), h_2(F(z))),
$$
where
$$
h_2(F(z))=\limsup_{n\to +\infty} d_1(g(n), o) - d_2(f(g(-n)), f(o)).
$$
Thus, for any bi-infinite geodesic $g$ in $X_1 $,
$$
F\big(\lim_{n\to +\infty} (g(-n), g(n))\big)=\lim_{k\to +\infty} (g(-n_k), f(g(n_k))) 
$$
for some diverging subsequence $(n_k)_k$.
This map $F\colon \Omega_1\to\Omega_2$ may  be non continuous. Moreover, if we replace the 
sequence $(n)_n$ in the definition of $F$ by another diverging 
sequence of positive reals, or if we choose another equivariant almost-isometry $f$, 
the definition of $F$ will differ, but the coarse results below will 
still hold.

\begin{lemma}
\label{lemma:bounds}
Let $K$ be the constant of almost isometry of $f$. Then:
 \begin{enumerate}[(i)]
  \item $\vert h_2(F(z))-h_1(z)\vert \leq K$, $\forall z\in\Omega_1$.
  \item For $i=1,2$, if $z,z'\in \Omega_i$ satisfy $\xi_i(z)=\xi_i(z')$ and $\eta_i(z)=\eta_i(z')$, then, $\forall\gamma\in\Gamma$,
  $$
  h_i(\gamma z)-h_i(\gamma z')=h_i( z)-h_i(z') .  
  $$
  \item $\vert h_2(F(\gamma z))-h_2(\gamma F(z))\vert \leq 4K$, $\forall z\in\Omega_1$, $\forall\gamma\in\Gamma$.
 \end{enumerate}
\end{lemma}

\begin{proof} (i)
Write $z\in \Omega_1$ as the limit $z=\lim_{n\to+\infty} (g(-n), g(n))$ for a (unique) geodesic $g$ in $\Omega_1$. Then
\begin{align*}
 h_1(z) & = \lim_{n\to+\infty} d(g(-n), o)-d(g(n),o), \\
 h_2(F(z)) & = \limsup _{n\to+\infty} d(g(-n), o)-d(f(g(n)),f(o)).
\end{align*}
From these expressions we get
$$
\vert h_2(F(z))-h_1(z)\vert\leq  \limsup _{n\to+\infty} \vert d(f(g(n)),f(o)) - d(g(n),o) \vert \leq K.
$$

(ii) We prove it for $i=2$, as the proof for $i=1$ is analogous. By Lemma~\ref{diagonalonboundary}: 
$$
h_2(\gamma z)-h_2(z)=\beta_{\xi_2(z)}^o(\gamma^{-1} o)- \beta_{\eta_2(z)}^{f(o)}(\gamma^{-1} f(o)).
$$
As we assume $\xi_2(z)=\xi_2(z')$ and $\eta_2(z)=\eta_2(z')$,   assertion (ii) is proved. 

(iii) We write:
\begin{multline*}
h_2(F(\gamma z))-h_2(\gamma F(z))=
\big( h_2(F(\gamma z))- h_1(\gamma z)\big)  +\big( h_1(\gamma z)-h_1(z)   \big)  \\ \qquad\qquad  \qquad\qquad  \qquad\qquad 
  + \big( h_1(z)-h_2(F(z))   \big) +  \big(  h_2(F(z)) -  h_2(\gamma F(z))  \big)
\\
=(I)+(II)+(III)+(IV)
. 
\end{multline*}
The terms $(I)$ and $(III)$ are bounded in absolute value by $K$ by item (i). By Lemma~\ref{diagonalonboundary}:
\begin{align*}
(II)&= h_1(\gamma z)-h_1(z)= \beta^o_{\xi_+}(\gamma^{-1} o)-\beta^o_{\xi_-}(\gamma^{-1} o), \\
(IV)&=   h_2(F(z)) -  h_2(\gamma F(z))  = -\beta^o_{\xi_+}(\gamma^{-1} o) + \beta^{f(o)}_{f_{\infty}(\xi_-)}(\gamma^{-1} f(o)) .
\end{align*}
Hence 
\begin{equation}
\label{eqn:24}
(II)+(IV)= \beta^{f(o)}_{f_{\infty}(\xi_-)}(f(\gamma^{-1} o)) - \beta^o_{\xi_-}(\gamma^{-1} o).
\end{equation}
For $r\colon[0,+\infty)\to X_1$ the geodesic ray with $r(0)=o$ that converges to $\xi_-$: 
\begin{equation}
\label{eqn:betao}
\beta^o_{\xi_-}(\gamma^{-1} o)=\lim_{t\to+\infty} d_1(r(t),\gamma^{-1} o)- d_1(r(t), o).
\end{equation}
On the other hand, $f\circ r\colon [0,+\infty)\to X_2$ is a quasi-geodesic that converges to $f_{\infty}(\xi_-)$.
Since the visual compactification and the compactification by horofunctions coincide for a $\CAT(-1)$-space, 
there is a diverging sequence of times  $(t_k)\to+\infty$
such that
\begin{equation}
\label{eqn:betafo}
\beta^{f(o)}_{f_{\infty}(\xi_-)}(f(\gamma^{-1} o))= \lim_{k\to+\infty} d_2(f(r(t_k)),f(\gamma^{-1} o))- d_1(f(r(t_k)), f(o)).
\end{equation}
Since $f$ is a $K$-almost isometry, it follows from \eqref{eqn:24}, \eqref{eqn:betao} and~\eqref{eqn:betafo} that
$
\vert (II)+(IV)\vert\leq 2 K
$.
\end{proof}

\begin{lemma}
\label{lemma:6K}
 Let $z\in \Omega_1$ and $y\in\Omega_2$ be such that $\xi_1(z)=\xi_2(y)$, $f_{\infty}(\eta_1(z))=\eta_2(y)$, and $h_1(z)=h_2(y)$. Then 
 $$
 \vert h_1(\gamma z)-h_2(\gamma y)\vert \leq 6K,\qquad  \forall\gamma\in\Gamma.
 $$
\end{lemma}

\begin{proof} It is a consequence of  the following three inequalities:
 \begin{align*}
 \vert h_1(\gamma z)-h_2(F(\gamma z))\vert & \leq K ,\\
 \vert h_2(F(\gamma z ))- h_2(\gamma  F(z ) ) \vert & \leq 4 K ,\\
 \vert h_2(\gamma  F(z ) ) - h_2(\gamma  y ) \vert & =  \vert h_2( F(z ) ) - h_2( y ) \vert 
      = \vert h_2( F(z ) ) - h_1( z)\vert \leq K .
\end{align*}
Here we have used  Lemma~\ref{lemma:bounds}, item (i) for the first line, item (iii) for the second, and items (ii) and  (i) for the last one. 
\end{proof}

\begin{proposition}
\label{prop:omega2}
 The  action of $\Gamma$ on $\Omega_{2}$ is properly discontinuous and cocompact.
 
\end{proposition}

\begin{proof}
We prove proper discontinuity by showing that no two points in $\Omega_2$
are dynamically 
related. Recall that two points $x,y$ in a metric space $Z$ are dynamically related by 
$\Gamma$ if there exist a sequences
$(z_n)_n$ in $Z$ and $(\gamma_n)_n$ in $\Gamma$  such that
$z_n\rightarrow x$, $\gamma_n\rightarrow 
\infty$, and  $\gamma_nz_n\rightarrow y$, see \cite{Frances}. Proper discontinuity is equivalent to 
the property that any two points (possibly equal) are not dynamically related.

By contradiction, we assume $ y_\infty$ and $\overline{y_\infty}$ in $\Omega_2$ are dynamically related, and we will show that two points in $\Omega_1$ are dynamically related.
Namely, assume
that there exists a sequence $(y_n)_n$ in $\Omega_2$ and a diverging sequence $(\gamma_n)_n$ in $\Gamma$ such that
$
y_n\to y_\infty\in\Omega_2 \textrm{ and } \gamma_ny_n\to\overline{y_\infty}\in\Omega_2.
$
For each $n\in \mathbb{N}$ let $z_n\in \Omega_1$ be such that $\xi_1(z_n)=\xi_2(y_n)$, $f_\infty(\eta_1(z_n))=\eta_2(y_n)$, and 
$h_1(z_n)= h_2(y_n)$ (we have defined $\varphi_i=(\xi_i,\eta_i,h_i)$). Since $\phi_1$ and $\phi_2$ are homeomorphisms, 
$z_n\to z_\infty\in\Omega_1$. On the other hand, 
the coordinates $\xi_1(\gamma_n z_n)$  and $\eta_1(\gamma_n z_n)$ also converge and it remains to bound $\vert h_1(\gamma_nz_n)\vert$: 
by Lemma~\ref{lemma:6K}  $\vert h_1(\gamma_nz_n) -  h_2(\gamma_n y_n) \vert\leq 6 K$ and 
$ h_2(\gamma_n y_n) \to  h_2(\overline{ y_\infty})$. 

Next we prove cocompactness. Let $(y_n)_n$ be a sequence in $\Omega_2$. For every $n\in\mathbb{N}$ we consider $z_n\in\Omega_1$ as above: 
$\xi_1(z_n)=\xi_2(y_n)$, $\eta_1(z_n)=f_\infty(\eta_2(y_n))$, and 
$h_1(z_n)= h_2(y_n)$. As the action is cocompact in $\Omega_1$, there exists a sequence  $\gamma_n$ in $\Gamma$ such that $\gamma_nz_n$ converges, 
and all 
we need to prove is that $\vert h_2(\gamma_n y_n)\vert$ is bounded. This is a consequence of the inequality
 $\vert h_2(\gamma_n y_n)-h_1(\gamma_n z_n)\vert \leq  6 K$ (by Lemma~\ref{lemma:6K}) and that
$h_1(\gamma_n z_n)$ converges.
\end{proof}

Now we consider the action on the whole $X_1\times X_2\cup \Omega_2$. We need the following lemma:

\begin{lemma}
\label{lemma:omega12}
Let $(x_n,y_n)_n$ be a diverging sequence in $X_1\times X_1$. The accumulation set of $(x_n,y_n)_n$  is contained in 
 $\Omega_1$ if and only if the accumulation set of $(x_n,f(y_n))_n$ is contained  in $\Omega_2$.

\end{lemma}

\begin{proof}
First assume that  $(x_n,y_n)_n$  converges to a point in $ \Omega_1$. Namely 
 $x_n\to \xi\in\partial_{\infty} X_1$, $y_n\to \eta\neq \xi\in\partial_{\infty} X_1$
and $\vert d_1(x_n,o)-d_1(y_n,o)\vert $ is bounded. Thus, as $x_n\to \xi$ and $f(y_n) \to f_\infty(\eta)\neq f_\infty(\xi)$, the assertion 
follows from the estimate
\begin{multline*}
| d_1(x_n,o)- d_2( f(y_n), f_n(o) ) | \\ \leq  | d_1(x_n,o)- d_1(y_n,o) |  + | d_1(y_n,o)- d_2( f(y_n), f_n(o) ) | ,
\end{multline*}
that is bounded because $f$ is $K$-almost isometry.
 
 For the converse, assuming that  $\vert d_1(x_n,o)-d_2(y_n,f(o))\vert $ is bounded, we write:
 \begin{multline*}
|d_1(x_n, o)- d_1(f^{-1}(y_n), o)|\\\leq |d_1(x_n, o) - d_2(y_n, f(o))| + |d_2(y_n, f(o)) -d_1(f^{-1}(y_n), o)|, 
\end{multline*}
that is bounded because:
\begin{multline*}
|d_2(y_n, f(o)) -d_1(f^{-1}(y_n), o)|\leq  |d_2(y_n, f(o))- d_1(f^{-1}(y_n), 
f^{-1}(f(o)))|\\+|d_1(f^{-1}(y_n), f^{-1}(f(o))) -d_1(f^{-1}(y_n), o)| \\\leq  
|d_2(y_n, f(o))- d_1(f^{-1}(y_n), f^{-1}(f(o)))| + d_1(f^{-1}(f(o)), o) \leq 2K.\qedhere
\end{multline*}
\end{proof}

\begin{theorem} The  action of $\Gamma$ on $X_1\times X_2\cup \Omega_2$ is properly discontinuous and cocompact.
\end{theorem}
\begin{proof} 
For proper discontinuity we will prove that  
no two points in $X_1\times X_2\cup \Omega_2$ are dynamically 
related, as in the proof of Proposition~\ref{prop:omega2}. 
Since the 
action is properly discontinuous on both $X_1\times X_2$ and 
$\Omega_2$, it is enough to check that if $(x_n, y_n)$ is a 
sequence in $X_1\times X_2$ that converges to a point
$z\in\Omega_2$, then there is no divergent sequence 
$(\gamma_n)_n\subset \Gamma$ such that $(\rho_1(\gamma_n)x_n, 
\rho_2(\gamma_n)y_n)$ accumulates in $X_1\times X_2\cup\Omega_2$. 
By contradiction, assume that such a divergent sequence 
$(\gamma_n)_n\subset \Gamma$ exists. 
If $(\rho_1(\gamma_n)x_n, \rho_2(\gamma_n)y_n)$ converges to a point $(x, 
y)\in X_1\times X_2$, then
$$
d_{\max} ( (\rho_1(\gamma_ n^{-1})(x), 
\rho_2(\gamma_ n^{-1}) (y)), (x_n, y_n))
$$
is uniformly bounded and 
$(\rho_1(\gamma_ n^{-1})(x), 
\rho_2(\gamma_ n^{-1}) (y))_n$ converges to the same point as $(x_n,y_n)_n$. Hence $z\in\Omega_2 $ is the accumulation point of an orbit
and we get a contradiction with Remark~\ref{remark_largelimitset2}. 

Therefore, we assume 
that  $(\rho_1(\gamma_n)x_n, \rho_2(\gamma_n)y_n)$ accumulates in 
$\Omega_2$. By Lemma~\ref{lemma:omega12}, both sequences 
$(x_n, f^{-1}(y_n))$ and  $(\rho_1(\gamma_n)x_n, \rho_1(\gamma_n)f^{-1}(y_n))$
accumulate in $\Omega_1$, which contradicts that $\Gamma$ acts properly discontinuously
on $X_1\times X_1\cup \Omega_1$.

To prove cocompactness and using Proposition~\ref{prop:omega2}, consider a sequence $(x_n,y_n)$ in
$X_1\times X_2$. There exists a sequence $\gamma_n$ of elements in $\Gamma$ such that 
$(\rho_1(\gamma_n) (x_n),\rho_1(\gamma_n) (f^{-1}(y_n)))$ accumulates in $X_1\times X_1\cup \Omega_1$.
Again by Lemma~\ref{lemma:omega12} $(\rho_1(\gamma_n) (x_n),\rho_2(\gamma_n) (y_n))$ accumulates in
$X_1\times X_2\cup \Omega_2$.
\end{proof}

\noindent\textbf{Acknowledgement.}
We are indebted to the anonymous referee for suggestions that have improved substantially the paper.

\end{document}